\documentclass[11pt]{amsart}
\usepackage{amsmath,amssymb}
\usepackage{verbatim}
\usepackage{color}
\usepackage{hyperref}

\begin{document}

	\newtheorem{thm}{Theorem}[section]
	\newtheorem{theorem}{Theorem}[section]
	\newtheorem{lem}[thm]{Lemma}
	\newtheorem{lemma}[thm]{Lemma}
	\newtheorem{prop}[thm]{Proposition}
	\newtheorem{proposition}[thm]{Proposition}
	\newtheorem{corollary}[thm]{Corollary}
	\newtheorem{definition}[thm]{Definition}
	\newtheorem{remark}[thm]{Remark}
	\newtheorem{conjecture}[theorem]{Conjecture}
	
	\numberwithin{equation}{section}
	
	\newcommand{\Z}{{\mathbb Z}} 
	\newcommand{\Q}{{\mathbb Q}}
	\newcommand{\R}{{\mathbb R}}
	\newcommand{\C}{{\mathbb C}}
	\newcommand{\N}{{\mathbb N}}
	\newcommand{\FF}{{\mathbb F}}
	\newcommand{\fq}{\mathbb{F}_q}
	\newcommand{\X}{{\mathbb {X}}}
	\newcommand{\rmk}[1]{\footnote{{\bf Comment:} #1}}
	
	\newcommand{\bfA}{{\boldsymbol{A}}}
	\newcommand{\bfY}{{\boldsymbol{Y}}}
	\newcommand{\bfX}{{\boldsymbol{X}}}
	\newcommand{\bfZ}{{\boldsymbol{Z}}}
	\newcommand{\bfa}{{\boldsymbol{a}}}
	\newcommand{\bfy}{{\boldsymbol{y}}}
	\newcommand{\bfx}{{\boldsymbol{x}}}
	\newcommand{\bfz}{{\boldsymbol{z}}}
	\newcommand{\F}{\mathcal{F}}
	\newcommand{\Gal}{\mathrm{Gal}}
	\newcommand{\Fr}{\mathrm{Fr}}
	\newcommand{\Hom}{\mathrm{Hom}}
	\newcommand{\GL}{\mathrm{GL}}
	\newcommand{\Res}{\mathrm{Res}}

	\renewcommand{\mod}{\;\operatorname{mod}}
	\newcommand{\ord}{\operatorname{ord}}
	\newcommand{\TT}{\mathbb{T}}
	\renewcommand{\i}{{\mathrm{i}}}
	\renewcommand{\d}{{\mathrm{d}}}
	\renewcommand{\^}{\widehat}
	\newcommand{\HH}{\mathbb H}
	\newcommand{\Vol}{\operatorname{vol}}
	\newcommand{\area}{\operatorname{area}}
	\newcommand{\tr}{\operatorname{tr}}
	\newcommand{\norm}{\mathcal N} 
	\newcommand{\intinf}{\int_{-\infty}^\infty}
	\newcommand{\ave}[1]{\left\langle#1\right\rangle} 
	\newcommand{\Var}{\operatorname{Var}}
	\newcommand{\Prob}{\operatorname{Prob}}
	\newcommand{\sym}{\operatorname{Sym}}
	\newcommand{\disc}{\operatorname{disc}}
	\newcommand{\CA}{{\mathcal C}_A}
	\newcommand{\cond}{\operatorname{cond}} 
	\newcommand{\lcm}{\operatorname{lcm}}
	\newcommand{\Kl}{\operatorname{Kl}} 
	\newcommand{\leg}[2]{\left( \frac{#1}{#2} \right)}  
	\newcommand{\Li}{\operatorname{Li}}

	\newcommand{\sumstar}{\sideset \and^{*} \to \sum}
	
	\newcommand{\LL}{\mathcal L} 
	\newcommand{\sumf}{\sum^\flat}
	\newcommand{\Hgev}{\mathcal H_{2g+2,q}}
	\newcommand{\USp}{\operatorname{USp}}
	\newcommand{\conv}{*}
	\newcommand{\dist} {\operatorname{dist}}
	\newcommand{\CF}{c_0} 
	\newcommand{\kerp}{\mathcal K}
	
	\newcommand{\Cov}{\operatorname{cov}}
	\newcommand{\Sym}{\operatorname{Sym}}
	
	\newcommand{\ES}{\mathcal S} 
	\newcommand{\EN}{\mathcal N} 
	\newcommand{\EM}{\mathcal M} 
	\newcommand{\Sc}{\operatorname{Sc}} 
	\newcommand{\Ht}{\operatorname{Ht}}
	
	\newcommand{\E}{\operatorname{E}} 
	\newcommand{\sign}{\operatorname{sign}} 
	
	\newcommand{\divid}{d} 
	\newcommand{\A}{{\mathbb{A}}}
	\newcommand{\h}{\mathbb{H}_{2g+1}}
	\newcommand{\p}{\mathbb{P}_{2g+1}}
	\newcommand{\f}{\mathbb{F}_{q}[T]}
	\newcommand{\z}{\zeta_{\A}}
	\newcommand{\lo}{\log_q}
	\newcommand{\x}{\chi}
	\newcommand{\xx}{\mathcal{X}}
	\newcommand{\lL}{\mathcal{L}}
	\newcommand{\e}{\varepsilon}
	\newcommand{\w}{\omega}
	\newcommand{\pp}{\text{\textbf{P}}}	
	\newcommand{\mm}{\mathcal{M}}
	\newcommand{\prof}{\textbf{Proof.}}	
	\newcommand{\RR}{\text{Re}}	
	\newcommand{\II}{\text{Im}}
	\newcommand{\D}{\mathcal{D}}

	\title[$n$-level Densities]
	{Correlations of zeros of a family of $L$-functions in function fields with symplectic symmetry}
	
	\author{Julio Andrade}
	\address{Department of Mathematics and Statistics, University of Exeter, Exeter, EX4 4QF, United Kingdom}
	\email{j.c.andrade@exeter.ac.uk}
	
	\author{ASMAA SHAMESALDEEN}
	\address{Department of Mathematics, The Public Authority for Applied Education and Training, Kuwait}
	\email{aa.shamsalldin@paaet.edu.kw}

	
	\begin{abstract}
	 In this paper, we adapt the framework developed by Mason and Snaith to investigate the $n$-level density of zeros in the context of function fields. Specifically, we derive explicit formulas for the $n$-level density of zeros in families of quadratic Dirichlet $L$-functions associated with hyperelliptic curves of genus $g$ over $\mathbb{F}_q$. Employing Mason and Snaith’s method, we obtain precise expressions for the $1$-level density in these families and extend the approach to higher-level densities. Furthermore, we apply the method to derive formulas for the $n$-level density of zeros in families of $L$-functions associated with prime characters. Our results are consistent with the findings of Andrade, Jung, and Shamesaldeen in the case $n=1$.
	\end{abstract}
	\date{\today}
	
	\maketitle

	
	\section{Introduction}\label{into}

	A common theme in analytic number theory is the study of the distribution of zeros of $L$-functions, and one of the main reasons for that is that such statistical quantities contain arithmetic information which appear in many parts of number theory. In the 1970s, Montgomery \cite{Montgomery} and Dyson \cite{Dyson} introduced groundbreaking ideas that connected analytic number theory with random matrix theory. Specifically, Montgomery \cite{Montgomery}, under the assumption of the Riemann Hypothesis, investigated the pair correlation of the zeros of the Riemann zeta function. His work revealed that, for test functions whose Fourier transforms are sufficiently restricted, this pair correlation exhibits a remarkable agreement with the eigenvalue statistics of random matrices drawn from the Gaussian Unitary Ensemble (GUE). 
	
	In the 1990s, Katz and Sarnak \cite{Kats&Sarnak1,Kats&Sarnak2} proposed a significant programme suggesting that each family of $L$-functions has an associated symmetry group. They investigated the statistical properties of zeros of $L$-functions near the critical line, where complex zeros are expected to lie. Their findings indicated that within a family of $L$-functions, these zeros exhibit statistical behavior similar to the eigenvalues of randomly selected matrices from the classical compact groups $U(N), O(N)$ or $US_p(2N)$, with respect to the Haar measure. In the appropriate limit, the zero statistics converge to the eigenvalue statistics of large matrices as a key parameter of the family grows indefinitely. Their conjecture has since inspired extensive research, including studies of leading-order statistical behavior \cite{1,6,2,7,3,4,Rudnick&Sarnak,5} and investigations into lower-order terms using the Ratios Conjecture \cite{8,9,10,11,12}. 
	
	Bogomolny and Keating \cite{BK1,BK2} furthered the study of the Riemann zeta function's zeros by investigating its $n$-point correlation functions, achieving the significant advancement of incorporating lower-order terms into their analysis. Their calculations crucially relied on the Hardy-Littlewood conjecture concerning twin primes and a detailed examination of the influence of prime numbers on the lower-order terms within the statistical distribution of the Riemann zeros. 
	
    Building upon these findings, Rudnick and Sarnak \cite{Rudnick&Sarnak} extended the analysis of two-point and three-point correlations to the general case of $n$-point correlations. By imposing specific constraints on the support of the Fourier transform of the test function, they were able to establish that the $n$-point correlation statistics of the zeros of principal $L$-functions exhibit a direct correspondence with the analogous theorems in Random Matrix Theory. 
	
    In 2000, Keating and Snaith \cite{Keating-Snaith1, Keating-Snaith2} proposed a seminal conjecture concerning the asymptotic behavior of the moments of the Riemann zeta function. By modeling the zeta function using the distribution of characteristic polynomials of random unitary matrices from the Circular Unitary Ensemble (CUE), they were able to put forward a conjecture for all the moments of the Riemann zeta function and other $L$-functions. This approach significantly deepened our understanding of the profound connections between $L$-functions and random matrices. 
	
    While there is substantial agreement between predictions derived from $L$-functions and those from the characteristic polynomials of random matrices—particularly regarding main terms—lower-order terms in many asymptotic formulas remain less well understood across various families of $L$-functions. This is largely due to the influence of subtle arithmetic features, which can disrupt the otherwise universal behavior predicted by random matrix theory. In 2008, building on the methodology developed by Conrey et al. \cite{CFKRS}, Conrey, Farmer, and Zirnbauer \cite{CFZ} formulated the Ratios Conjecture, offering conjectural formulas for the averages of ratios of products of shifted $L$-functions. This framework provided a powerful and systematic tool for predicting both leading and lower-order terms in the statistical behavior of $L$-function zeros, significantly enhancing our ability to address complex problems such as the $n$-level density in families of $L$-functions. 
	
    Building on the conjectures and techniques developed in \cite{CFZ}, Conrey and Snaith \cite{Conrey&Snaith,Conrey&Snaith1,Conrey&Snaith2,Conrey&Snaith3} investigated a range of number-theoretic applications of the Ratios Conjecture. They applied this framework to analyze lower-order terms in the zero statistics of $L$-functions, mollified moments, and discrete averages over the zeros of the Riemann zeta function. Within the setting of $U(N)$ random matrix theory, they showed that restricting the support of the test function yields notable simplifications in the $n$-point correlation functions—an approach consistent with the methodology of Rudnick and Sarnak \cite{Rudnick&Sarnak}. Furthermore, they derived an asymptotic formula for the one-level density in the symplectic family of $L$-functions, providing deeper insight into the fine-scale distribution of zeros. 
	
	Huynh, Miller, and Morrison \cite{9}, along with Miller in \cite{10,11}, computed the one-level density for the symplectic family of quadratic Dirichlet $L$-functions, showing agreement with predictions from the Ratios Conjecture up to an error of $O(X^{-1/2+\varepsilon})$ for test functions supported in the intervals $(-1/3, 1/3)$ and $(-1, 1)$. Building on this work, Mason and Snaith \cite{mason&Snaith} extended these techniques to families of $L$-functions exhibiting orthogonal or symplectic symmetry. By reformulating the Ratios Conjecture using set notation, they derived explicit expressions for the $n$-level densities of zeros in these families. Their approach allows for flexible adjustments based on the support of the test function, enabling direct and precise comparisons with rigorous number-theoretic results for $n$-level densities.

    The body of research outlined above, along with contributions from many other authors, underscores the profound connections between $L$-functions and random matrix theory, emphasizing the importance of studying statistical properties of $L$-functions to uncover their underlying arithmetic and analytic structure.

In the function field setting, Andrade and Keating \cite{a&kConInMo} extended the Ratios Conjecture to quadratic Dirichlet $L$-functions associated with monic and square-free polynomials of degree $2g + 1$ in $\mathbb{F}_q[T]$. Their approach followed the well-established methodology introduced by Conrey et al. \cite{CFKRS,CFZ}. 
	
	\begin{conjecture}[Andrade and Keating]\label{Andrade&Keating Ratios Conjecture}\label{ratios}
		Suppose that the real parts of $\alpha_k$ and $\gamma_m$ are positive and that $q$ odd is the fixed cardinality of the finite field $\mathbb{F}_q$. Then we have
		
		\begin{equation*}
			\begin{split}
				\sum_{D\in\h}& \frac{\prod_{k=1}^KL\left(\tfrac{1}{2}+\alpha_k,\x_D\right)}{\prod_{m=1}^QL\left(\tfrac{1}{2}+\gamma_m,\x_D\right)}\\
				& =\sum_{D\in\h} \sum_{\varepsilon\in\{-1,1\}^K} |D|^{\frac{1}{2} \sum_{k=1}^K(\varepsilon_k-1)\alpha_k} \prod_{k=1}^{K} \xx\left(\tfrac{1}{2} +\tfrac{(1-\varepsilon_k)}{2}\alpha_k\right)\\
				&\ \ \ \times  Y_{D}(\varepsilon_1\alpha_1,\cdots,\varepsilon_K\alpha_K;\gamma) A_{D}(\varepsilon_1\alpha_1,\cdots,\varepsilon_K\alpha_K;\gamma)+o\left(|D|\right),
			\end{split}
		\end{equation*}
		where $\xx(s)=q^{-\frac{1}{2}+s}$ and $L(s,\chi_D)$ is the quadratic Dirichlet $L$-function associated to $D$, with $D$ being a monic and square-free polynomial in $\mathbb{F}_q[T]$,
		
		
		\begin{equation*}
			Y_D(\alpha;\gamma) = \frac{\underset{j\leqslant k\leqslant K}{\prod} \z(1+\alpha_j+\alpha_k) \underset{m<r \leqslant Q}{\prod} \z(1+\gamma_m+\gamma_r)} {\prod_{k=1}^K\prod_{m=1}^Q \z\left(1+\alpha_k+\gamma_r\right)},
		\end{equation*}
		and
		
		\begin{equation*}
			\begin{split}
				A_D(\alpha,\gamma) &= Y_D(\alpha;\gamma)^{-1}  \prod_{\substack{P \text{ monic}\\\text{irreducible}}}\left(1+\frac{1}{|P|}\right)^{-1} \\
				&\ \ \ \times  \left(\frac{1}{2} \frac{\prod_{q=1}^Q\left(1-\frac{1}{|P|^{\frac{1}{2}+\gamma_q}}\right)} {\prod_{k=1}^K\left(1-\frac{1}{|P|^{\frac{1}{2}+\alpha_k}}\right)}+ \frac{1}{2} \frac{\prod_{q=1}^Q\left(1+\frac{1}{|P|^{\frac{1}{2}+\gamma_q}}\right)} {\prod_{k=1}^K\left(1+\frac{1}{|P|^{\frac{1}{2}+\alpha_k}}\right)} +\frac{1}{|P|}\right),
			\end{split}
		\end{equation*}
		and $\zeta_{\mathbb{A}}(s)$ is the zeta-function associated to $\mathbb{F}_q[T]$.
	\end{conjecture}
	
	Using the Ratios Conjecture, Andrade and Keating computed the one-level density for this family of quadratic Dirichlet $L$-functions.
	
	 	\begin{theorem}\label{andrade&keating n=1}
	 	Assuming the ratios Conjecture \ref{Andrade&Keating Ratios Conjecture}, we have that
	 	
	 	\begin{equation*}
	 		\begin{split}
	 			S_1(f)&= \sum_{D\in\h} \sum_{\gamma_D} f(\gamma_D) =  \frac{1}{2\pi} \int_{-\pi/\ln q}^{\pi/\ln q} f(z) \sum_{D\in \h} \Bigg[ \log|D|+\frac{\xx'}{\xx}(\frac{1}{2}-iz) \\
	 			& \ \ \ + 2  \Bigg(  \frac{\z'}{\z}(1+2iz)+ A_D'(iz;iz) - \left(\ln q\right) |D|^{iz} \xx\left(\tfrac{1}{2}+iz\right) \z(1-2iz)  \\
	 			& \ \ \ \times A_D(-iz,iz)\Bigg) \Bigg]dz + o\left(|D|\right),
	 		\end{split}
	 	\end{equation*}
	 	where $\gamma_D$ is the ordinate of a generic zero of $L(s,\x_D)$ and $f$ is an even and periodic suitable test function,
	 	
	 	\begin{equation}\label{AD}
	 		\begin{split}
	 			A_D(\alpha;\gamma)&= \prod_{\substack{P \ \text{monic} \\ \text{irreducible}}} \left(1-\frac{1}{|P|^{1+\alpha+\gamma}}\right)^{-1} \\
				& \; \; \;\ \; \; \times  \left(1-\frac{1}{(|P|+1)|P|^{1+2\alpha}}-\frac{1}{(|P|+1)|P|^{\alpha+\gamma}}\right),\\
	 			A_D(r;r)&=1,\\
	 			A_D(-r;r)&= \prod_{\substack{P \ \text{monic} \\ \text{irreducible}}} \left(1-\frac{1}{|P|}\right)^{-1} \left(1-\frac{1}{(|P|+1)|P|^{1-2r}}-\frac{1}{|P|+1}\right),\\
		A_D'(r;r)&= \sum_{\substack{P \ \text{monic} \\ \text{irreducible}}}  \frac{\log|P|}{(|P|^{1+2r}-1)(|P|+1)},
	 		\end{split}
	 	\end{equation}
	 \end{theorem}
	
	 In a related study, Andrade et al. \cite{AMPT} examined the function field analogue of the Dirichlet $L$-functions previously studied by Hughes and Rudnick \cite{HR}. By employing a suitably periodic test function, they computed the one- and two-level statistics of zeros in the global regime. Their results, which rely only on mild decay conditions on the Fourier coefficients, established a more general framework for analyzing the statistical distribution of zeros of $L$-functions in the function field context.

Further contributions were made by Bui and Florea \cite{BF}, who examined the one-level density and pair correlation of zeros for the same family of quadratic Dirichlet $L$-functions in function fields. They demonstrated that when the Fourier transform of the test function is supported in the restricted interval $\left(-1/3,1/3\right)$, a secondary term of order $q^{-4g/3}$ appears, which is a term absent in the predictions of Andrade and Keating Ratios Conjecture. Moreover, they showed that with further restrictions on the support of the test function, additional lower-order terms could be identified.
	
	 For $L$-functions associated with the prime character $\x_P,$ where $P$ is a monic irreducible polynomial of degree $2g + 1$ in $\mathbb{F}_q[T]$, Andrade, Jung, and Shamesaldeen \cite{Andrade Jung Shames} extended the function field analogy used in \cite{a&kConInMo} and formulated conjectures for moments and ratios of this family of $L$-functions.
	
	 \begin{conjecture}[Andrade, Jung and Shamesaldeen]\label{rationconjecture}
	 	Suppose that the real parts of $\alpha_k$ and $\gamma_m$ are positive and that $q\equiv 1 (\mod 4)$ is the fixed cordiality of the finite field $\mathbb{F}_q.$ Then we have,
	 	
	 	\begin{equation}
	 		\begin{split}
	 			\sum_{P\in\mathbb{P}_{2g+1}}& \frac{\prod_{k=1}^KL(\frac{1}{2}+\alpha_k,\x_P)}{\prod_{m=1}^QL(\frac{1}{2}+\gamma_m,\x_P)}\\
	 			& =\sum_{P\in\mathbb{P}_{2g+1}} \sum_{\varepsilon\in\{-1,1\}^K} |P|^{\frac{1}{2}\sum_{k=1}^K (\varepsilon_k-1)\alpha_k)} \prod_{k=1}^K \xx\left(\tfrac{1}{2}+\tfrac{\alpha_k-\varepsilon_k\alpha_k}{2}\right)\\
	 			& \ \ \	\times  Y_{\mathfrak{P}}(\varepsilon_1\alpha_1,\cdots,\varepsilon_K\alpha_K;\gamma)A_{\mathfrak{P}}(\varepsilon_1\alpha_1,\cdots,\varepsilon_K\alpha_K;\gamma)+o\left(|P|\right).\\
	 		\end{split}
	 	\end{equation}
	 	where
	 	
	 	\begin{equation*}
	 		Y_{\mathfrak{P}}(\alpha;\gamma):= \frac{\prod_{ j\leqslant k\leqslant K} \z(1+\alpha_j+\alpha_k)\prod_{ q<r\leqslant Q} \z(1+\gamma_r+\gamma_q)}{\prod_{k=1}^k\prod_{q=1}^Q \z(1+\alpha_k+\gamma_q)}.
	 	\end{equation*}
	 	and
	 	\begin{equation*}
	 		\begin{split}
	 			A_{\mathfrak{P}}(\alpha;\gamma)&= Y_{\mathfrak{P}}(\alpha;\gamma)^{-1}\\
	 			&\ \ \ \times \prod_{\substack{P \text{ monic}\\ \text{irreducible}}}  \frac{1}{2} \left(\frac{\prod_{q=1}^Q\left(1-\frac{1}{|P|^{\frac{1}{2}+\gamma_q}}\right)} {\prod_{k=1}^K\left(1-\frac{1}{|P|^{\frac{1}{2}+\alpha_k}}\right)}+\frac{\prod_{q=1}^Q\left(1+\frac{1}{|P|^{\frac{1}{2}+\gamma_q}}\right)} {\prod_{k=1}^K\left(1+\frac{1}{|P|^{\frac{1}{2}+\alpha_k}}\right)}\right).
	 		\end{split}
	 	\end{equation*}
	 \end{conjecture}
	
	 By using this conjecture, the authors computed the one-level density for this corresponding family of $L$-functions.
	 		
	 \begin{theorem}
	 	Assuming the ratios Conjecture \ref{rationconjecture}, we have that
	 	
	 	\begin{equation}\label{a.j.sh.1-level}
	 		\begin{split}
	 			S_1(f)=& \sum_{P\in\p} \sum_{\gamma_P} f(\gamma_P)\\
	 			=&  \frac{1}{2\pi} \int_{-\pi/\log q}^{\pi/\log q} f(t) \sum_{P\in \p} \Bigg( \log|P|+\frac{\xx'(\frac{1}{2}-it)}{\xx(\frac{1}{2}-it)} \\
	 			& + 2  \Bigg(  \frac{\z'(1+2it)}{\z(1+2it)}- \left(\log q\right) |P|^{-it} \xx\left(\tfrac{1}{2}+r\right) \z(1-2it)\Bigg) \Bigg)dt \\
	 			& + o\left(|P|\right),
	 		\end{split}
	 	\end{equation}
	 	where $\gamma_P$ is the ordinate of a generic zero of $L(s,\x_P)$ and $f$ is an even and periodic suitable test function.
	 \end{theorem}

The primary goal of this paper is to advance the study of $n$-level densities of zeros of $L$-functions in the function field setting and to generalize some of the results discussed above. In particular, we develop the function field analogue of the methodology introduced by Mason and Snaith \cite{mason&Snaith}, focusing on the family of quadratic Dirichlet $L$-functions associated with hyperelliptic curves of genus $g$ over a fixed finite field $\mathbb{F}_q$. Adapting their approach, we apply the Ratios Conjecture \ref{ratios} to compute the $n$-level density for this family, providing new insights into the statistical behavior of their low-lying zeros. 
	
	
	\section{Preliminaries on Dirichlet $L$-functions in Function Fields}\label{background}
	
	This section will give some basic facts and notations about $L$-functions in function fields. Let $q$ be fixed and $q \equiv 1(\mod 4)$, $\mathbb{F}_q$ a finite field with $q$ elements, and $\A:= \f$ be the polynomial ring over $\mathbb{F}_q$. The norm of a polynomial $f\in\f$ is defined as $|f| = q^{\deg(f)}$ for $f\neq 0$ and $0$ otherwise. The letter $P$ will be used to denote a monic and irreducible polynomial. 
	
	The zeta function of $\A$, is denoted by $\z(s)$  and defined by the infinite series
	\begin{equation*}
		\z(s):=\sum_{f  \text{ monic}} \frac{1}{|f|^s} = \prod_{\substack{P \text{ monic} \\\text{irreducible}}} \left(1-\frac{1}{|P|^s}\right)^{-1}.
	\end{equation*}
	Since there are $q^k$ different monic polynomials of degree $k$, it is easy to show that
	
	\begin{equation*}
		\z(s)=\frac{1}{1-q^{1-s}}.
	\end{equation*}
	The analogue of the M\"{o}bius function, $\mu(f)$,  for $\A$ is defined as follows
	
	\begin{equation*}
		\mu(f) = \begin{cases}
			(-1)^t & \text{ if } f=\alpha P_1 P_2 \cdots P_t \\
			0 & \text{ Otherwise }
		\end{cases}
	\end{equation*}
	where each $P_j$ is a distinct monic irreducible polynomial. 
	
	For a monic irreducible polynomial $P\in\A$, the quadratic residue symbol $\left(\frac{f}{P}\right)$, with $f$ coprime to $P$, is defined by
	\begin{equation*}
		\left(\frac{f}{P}\right)\equiv f^{\frac{|P|-1}{2}} \left(\mod P\right).
	\end{equation*}
	
	We define the Jacobi symbol $\left(\frac{f}{Q}\right)$ for arbitrary monic polynomial $Q$, where $f$ is coprime to $f$ and $Q= P^{e_1}_1 \cdots P^{e_k}_k$, by
	\begin{equation*}
		\left(\frac{f}{Q}\right) = \prod_{i=1}^k \left(\frac{f}{P_i}\right)^{e_i}.
	\end{equation*}
	
	Let $D\in\f$ be square-free. We define the quadratic character $\x_D$ by
	
	\begin{equation*}
		\x_D(f)=\left(\frac{D}{f}\right).
	\end{equation*}
	That is, if $P\in\f$ is a monic and irreducible polynomial we have
	\begin{equation*}
		\x_D(P)=\begin{cases}
			0, & \text{ if } P|D,\\
			1, & \text{ if } P\nmid D \text{ and } D \text{ square modulo } P,\\
			-1, & \text{ if } P\nmid D \text{ and } D \text{ non-square modulo } P.
		\end{cases}
	\end{equation*} 
	
	We define the $L$-function corresponding to the quadratic character $\x_D$ by
	
	\begin{equation*}
		L(s,\x_D) = \sum_{f \text{ monic}} \frac{\x_D(f)}{|f|^s} = \prod_{\substack{P \text{ monic} \\ \text{ irreducible}}} \left(1-\frac{\x_D(P)}{|P|^s}\right)^{-1}, \text{ for } \RR(s)>1.
	\end{equation*}
	With the change of variable $u=q^{-s}$ we have
	
	\begin{equation*}
		\begin{split}
			\mathcal{L}(u,\x_D) :&= L(s,\x_D) = \sum_{f \text{ monic}} \x_D(f)u^{\deg(f)}\\
			& =\prod_{\substack{P \text{ monic} \\ \text{ irreducible}}} \left(1-\x_D(P)u^{\deg(P)}\right)^{-1}.
		\end{split}
	\end{equation*}
	
	Let $\mathbb{H}_d$ be the set of all square-free monic polynomials of degree $d$ in $\f$.  The cardinality of $\mathbb{H}_d$ is
	\begin{equation*}
		\#\mathbb{H}_d=\begin{cases}
			\left(1-\frac{1}{q}\right)q^d, & d\geq 2, \\
			q, & d=1.
		\end{cases}
	\end{equation*}
	Therefore, when $D\in\h$ and $g\ge 1$ we have
	
	\begin{equation*}
		\h= \left(q-1\right)q^{2g} = \frac{|D|}{\z(2)}.
	\end{equation*}
	
	Finally, if we let $D\in\h$ we have the ``approximate" functional equation, first presented in \cite[Lemma 3.3]{a&kmeanvalue}, for the $L$-function given by
	
	\begin{equation}\label{aprox}
		L(s,\x_D) = \sum_{\substack{f_1 \text{ monic} \\ \deg(f_1)\le g}} \frac{\x_D(f_1)}{|f_1|^s} +\xx_D(s) \sum_{\substack{f_2 \text{ monic} \\ \deg(f_2)\le g-1}} \frac{\x_D(f_1)}{|f_2|^{1-s}},
	\end{equation}
	where $\xx_D(s) =q^{g(1-2s)}$. See \cite{a&kmeanvalue} and \cite{Rosen} for more details and background information on $L$-functions in function fields.

	\section{The Statement of the Main Results}\label{Results}
	
	The main aim of this paper is to write down a conjecture for the $n$-level density for the family of quadratic Dirichlet $L$-functions attached to square-free monic polynomials in $\mathbb{F}_q[T]$. In other words, our main aim is to adapt the work of Mason and Snaith \cite{mason&Snaith} to the function field setting. In this section, we give a brief summary of the paper and of what you may expect in the sections ahead.
	
	In section \ref{The Ratios Conjecture}, we re-write Andrade and Keating's ratios conjecture \ref{Andrade&Keating Ratios Conjecture} in a set form by defining two sets of shifts $A=\{\alpha_1,\cdots,\alpha_K\}$ and $B=\{\gamma_1\cdots,\gamma_Q\}$ of complex numbers such that
	
	\begin{equation}\label{comditions}
		\begin{array}{cc}
			-\frac{1}{4}<\RR(\alpha_k)<\frac{1}{4}, \ \text{where $1\leq k\leq K$},   \\
			\frac{1}{\log |P|}\ll \RR(\gamma_m)<\frac{1}{4} \ \text{where $1\leq m\leq Q$},  \\
			\forall \varepsilon>0, \ \II (\alpha_k),\II(\gamma_m) \ll_\varepsilon |D|^{1-\varepsilon}.   \\
		\end{array}
	\end{equation}
	We also define the following products of zeta functions
	
	\begin{equation}\label{Zz}
		\begin{split}
			Z_{\z}(A,B)&= \prod_{\substack{\alpha\in A \\ \beta\in B}} \z(1+\alpha+\beta),\\
			Y_{\z}(A) &= \prod_{\substack{\alpha\in A}} \z(1+2\alpha),\\
		\end{split}
	\end{equation}
	and the Euler factors
	\begin{equation}\label{V_+} 
		\begin{split}
			V_{+}(A) & = \prod_{\alpha\in A} \left(1+\frac{1}{|P|^{\frac{1}{2}+\alpha}}\right),\\
			V_{-}(A) &= \prod_{\alpha\in A} \left(1-\frac{1}{|P|^{\frac{1}{2}+\alpha}}\right).\\	
		\end{split}
	\end{equation}
	Let $\D\subseteq A$ and $\D^{-} = \{-\delta : \delta \in \D\}$, we define
	
	\begin{equation}\label{A_{DL}}
		\begin{split}
			A_{DL}&(A,B,\D) = Y_{DL}(A,B,\D)^{-1} \\
			&  \times \prod_{\substack{P \text{ monic} \\ \text{irreducible}}} \left(1+\frac{1}{|P|}\right)^{-1}  \left(\frac{V_{-}(B)}{2V_{-}((A\setminus \D) \cup \D^{-})} +  \frac{V_{+}(B)}{2V_{+}((A\setminus \D) \cup \D^{-})}+\frac{1}{|P|}\right),
		\end{split}
	\end{equation}
	where
	
	\begin{equation}\label{Y_{DL}}
		\begin{split}
			Y_{DL}&(A,B,\D) \\
			&= \sqrt{\frac{Z_{\z}((A\setminus \D) \cup \D^{-},(A\setminus \D) \cup \D^{-})  Y_{\z}((A\setminus \D) \cup \D^{-}) Z_{\z}(B,B)}{Z_{\z}((A\setminus \D) \cup \D^{-},B)^2 Y_{\z}(B)}}.
		\end{split}
	\end{equation}
	Using the notation above in Andrade and Keating's ratios conjecture \ref{Andrade&Keating Ratios Conjecture} we are then able to re-write the ratios conjecture as the following.
	
	\begin{conjecture}\label{Our Ratios Conjecture2.}
		Let $A=\{\alpha_1,\cdots,\alpha_K\}$ and $B=\{\gamma_1\cdots,\gamma_Q\}$ such that $\alpha_k$ is such that, $1\leqslant k \leqslant K$, and $\gamma_m$ is such that, $1\leqslant m \leqslant Q$ satisfying the restrictions in (\ref{comditions}).
		Then
		
		\begin{equation}\label{df}
			\begin{split}
				R_{DL}(A,B)&=\sum_{D\in\h}  \frac{\prod_{k=1}^K L\left(\tfrac{1}{2}+\alpha_k,\x_D\right)}{\prod_{m=1}^Q L\left(\tfrac{1}{2}+\gamma_m,\x_D\right)}\\
				& = \sum_{D\in\h} \sum_{\D\subseteq A} \left|D\right|^{- \sum_{\delta\in \D}\delta } \prod_{\delta\in \D} \xx\left(\tfrac{1}{2}+\delta\right) Y_{DL}(A,B,\D) A_{DL}(A,B,\D)\\
				& \ \ \ + o\left(|D|\right),
			\end{split}
		\end{equation}
		where $Y_{DL}(A,B,\D),A_{DL}(A,B,\D)$ are defined as in equations (\ref{Y_{DL}}) and (\ref{A_{DL}}) respectively.
	\end{conjecture}
	
	In Section \ref{DifRatios}, we differentiate the ratios conjecture \ref{Our Ratios Conjecture2.} and obtain the following theorem.

	\begin{theorem}\label{J^*_{DL}.}
		Assume Conjecture \ref{Our Ratios Conjecture2.}. Let $A$ be a finite set of complex numbers such that for $\alpha\in A, \  \frac{1}{\log |D|}\ll \RR(\alpha)<\frac{1}{4}$ and $\forall \varepsilon>0, \ \Im(\alpha) \ll_\varepsilon |D|^{1-\varepsilon}.$ Then
		
		\begin{equation*}
			J_{DL}(A)= J^*_{DL}(A)+o(|D|)
		\end{equation*}
		where
		
		\begin{equation*}
			\begin{split}
				J_{DL}(A) 
				&= \prod_{\alpha\in A} \frac{\partial}{\partial\alpha} R_{DL}(A,B)\Big\vert_{B=A},
			\end{split}
		\end{equation*}
		and
		\begin{equation}\label{J_{DL(A)}}
			\begin{split}
				J^*_{DL}(A) &= \sum_{D\in\h} \sum_{\D\subseteq A} \left|D\right|^{- \sum_{\delta\in \D}\delta } \prod_{\delta\in \D} \xx\left(\tfrac{1}{2}+\delta\right)   \left( -\log q\right)^{\#\D}\\
				&\ \ \ \times \sqrt{\frac{	 Z_{\z}(\D^{-},\D^{-}) Z_{\z}(\D,\D) Y_{\z}(\D^{-})  }{  Z_{\z}^\dag( \D^{-},\D)^2  Y_{\z}(\D)}} \\
			\end{split}
		\end{equation}	
		\begin{equation*}
			\begin{split}
				\color{white}J^*_{DL}(A)
				& \ \ \ \times A_{DL}(\D,\D,\D) \sum_{\substack{A\setminus \D = W_1+\cdots+W_R\\ |W_r|\leqslant 2}} \prod_{r=1}^R \widetilde{H}^{A,B}_\D(W_r)
			\end{split}
		\end{equation*}	
		with $Z_{\z}(A,B), Y_{\z}(A),A_{DL}(A,B,\D)$ and $\widetilde{H}^{A,B}_\D(W_r)$ defined as in equations (\ref{Zz}), (\ref{A_{DL}}) and (\ref{widetilde{H}^{A,B}_D}) respectively.
	\end{theorem}
	
	Writing the ratios conjecture \ref{Andrade&Keating Ratios Conjecture} in a set form it helps with the differentiation as we can move the differentiation inside the sum over subsets $\D \subset A$, and then proceed by separating the differentiations in cases either $\alpha\in \D$ or $\alpha\notin \D$. Writing the conjecture using the sets $A$ and $B$ also helps when dealing with poles of $J^*_{DL}(A)$ and the residue calculation at these poles as can be seen in section \ref{Residue Theorem}. In the end, we obtain the following theorem.

	\begin{theorem}\label{Residue.}
		Let $A$ be a finite set of complex numbers, let $\alpha^*,\beta^*\in A$ and $A'=A\backslash\{\alpha^*,\beta^*\}$ and let $J^*_{DL}(A)$ as defined in Theorem \ref{J^*_{DL}.}. Then
		
		\begin{equation*}
			\begin{split}
				\underset{\alpha^*=-\beta^*}{\mathrm{Res}}&(J^*_{DL}(A)) \\
				& =J^*_{DL}(A'\cup \{\beta^*\}) + J^*_{DL}(A'\cup \{-\beta^*\}) -\frac{\xx'_{D}}{\xx_{D}} \left(1+\beta^*\right) J^*_{DL}(A')
			\end{split}
		\end{equation*}
		where $\frac{\xx'_{D}}{\xx_{D}}(s)=-\log|D|+\log q$ comes from the functional equation. We also have
		
		\begin{equation*}
			\underset{\alpha^*=0}{\mathrm{Res}}(J^*_{DL}(A))=0.
		\end{equation*}
	\end{theorem}
	
	In Section \ref{ndensity}, we derive the formula for the $n$-level density for the zeros of quadratic Dirichlet $L$-functions over function fields.
	
	\begin{theorem}[$n$-level Density of Zeros of Quadratic Dirichlet $L$-functions] \label{n-density theorem.}
		Assume Conjecture \ref{Our Ratios Conjecture2.}. Then
		
		\begin{equation*}
			\begin{split}
				\sum_{D\in\h} &\sum\nolimits^{n,0} f\left(\gamma_{j_1,D},\cdots,\gamma_{j_n,D}\right) \\
				& =\frac{1}{\left(2\pi\right)^n} \sum_{\substack{K\cup L\cup M \\=\{1,\cdots,n\}}} \left(-1\right)^{|M|} \left(\int_0^{\pi /\ln q}\right)^n J_{DL}^*\left(-iz_K\cup iz_L,iz_M\right) \\
				& \ \ \ \times f\left(z_1,\cdots,z_n\right) dz_1\cdots dz_n + o\left(|D|\right),
			\end{split}
		\end{equation*}
	where the sum $\sum\nolimits^{n,R} $ is defined in (\ref{n,R}), $z_K=\{z_k:k\in K\}$, $-z_L=\{-z_l:l\in L\}$, $-z_M=\{-z_m:m\in M\}$, and
			\begin{equation*}
				\begin{split}
					J^*_{DL}(A,B) &= \sum_{P\in\mathbb{P}_{2g+1}} \prod_{\beta\in B} \frac{\xx'_P}{\xx_P}\left(\tfrac{1}{2}-\beta,\x_P\right) \sum_{\D\subseteq A} \left|P\right|^{- \sum_{\delta\in \D}\delta }   \\
					&\ \  \ \times \prod_{\delta\in \D} \xx\left(\tfrac{1}{2}+\delta\right) \sqrt{\frac{ Z_{\z}(\D^{-},\D^{-}) Z_{\z}(\D,\D) Y_{\z}(\D^{-})  }{  Z_{\z}^\dag( \D^{-},\D)^2  Y_{\z}(\D)}} \\
					& \ \ \ \times \left( -\log q\right)^{|\D|} A_{DL}^*(\D,\D,\D)\\
					& \ \ \ \times \sum_{\substack{A\setminus \D = \sum_rW_r\\ |W_r|\leqslant 2}} \prod_{r=1}^R \widetilde{H}^{A,B}_\D(W_r),
				\end{split}
			\end{equation*}
			where
			\begin{equation*}\label{widetilde{H}^{A,B}_DP}
				\begin{split}
					\widetilde{H}^{A,B}_\D (W_r)&= \prod_{\alpha\in W_r}\frac{\partial}{\partial \alpha} H^{A,B}_\D\\
					&= H^{A,B}_\D(W_r)+A^{A,B}_\D(W_r)
				\end{split}
			\end{equation*}
			with
			
			%
			\begin{equation*}
				A^{A,B}_\D(W_r) = \prod_{\alpha\in W_r} \frac{\partial}{\partial \alpha} \log A_{DL}^* (A,B,\D) \Bigg\vert_{B=A},
			\end{equation*}
			
			\begin{equation*}\label{A_{DL}P}
				\begin{split}
					A_{DL}^*(A,B,\D) &= Y_{DL}(A,B,\D)^{-1} \\
					& \ \ \ \times \prod_{\substack{P \text{ monic}\\ \text{irreducible}}}  \frac{1}{2} \left(\frac{V_{-}(B)}{V_{-}((A\setminus \D) \cup \D^{-})} +  \frac{V_{+}(B)}{V_{+}((A\setminus \D) \cup \D^{-})}\right),
				\end{split}
			\end{equation*}
			and $V_+(A)$, $V_-(A)$ are defined in equation (\ref{V_+}),  $Y_{DL}(A,B,\D)$ and ${H}^{A,B}_\D(W_r)$ are defined in equations (\ref{Y_{DL}}) and (\ref{{H}^{A,B}_D}) respectively.
	\end{theorem}

	Finally, in Section \ref{1-level}, we test our formula for the $n$-level density of the zeros of  quadratic Dirichlet $L$-functions. We use the formula above to compute the 1-level density and compare it with Andrade and Keating's result in \cite{a&kConInMo} confirming that the results matches.

	\subsection{$n$-level Density over the Family of $L$-functions Associated to Monic Irreducible Polynomials}
	
	Before we end this section, we need to point out that the calculation in this paper had also been made for the family of $L$-functions associated with the prime character $P\in\f$.
	
	As mentioned in the introduction, Andrade, Jung, and Shamesldeen \cite{Andrade Jung Shames} deduced the ratios conjecture for the family of $L$-functions associated with the prime character. As an application of the conjecture, they computed the one-level density for the same family. The difference between their results and Andrade and Keating's results \cite{a&kConInMo}, for the family of quadratic Dirichlet $L$-functions, is the Euler product that appears in the quadratic family when averaging over all discriminants $D\in\h$. 
	
	The same happens when we compute the $n$-level density for the family of $L$-functions associated with the prime character $P$. The prime factor that appears in $A_{D}$, in (\ref{A_{DL}}), will not appear in the case of prime character, since when averaging over all primes $P\in\mathbb{P}_{2g+1}$ we get
	
	\begin{equation*}
		\begin{split}
			\lim_{g\to\infty} \frac{1}{\#\mathbb{P}_{2g+1}} \sum_{P\in\mathbb{P}_{2g+1}} \x_P(f) = \begin{cases}
				1 & \text{ if } f \text{ square} \\
				0 & \text{ if } f \text{ non-square.}
			\end{cases}
		\end{split}
	\end{equation*}
	Similarly, we can deduce the $n$-level density for the family of $L$-functions associated with prime character as follows:
	
	\begin{theorem}[$n$-level Density of Zeros of families of Prime Quadratic Dirichlet $L$-functions] \label{n-density theoremP}
		Let $A=\{\alpha_1,\cdots,\alpha_K\}$ and $B=\{\gamma_1\cdots,\gamma_Q\}$ such that $\alpha_i, 1\leqslant k \leqslant K$ and $\gamma_m, 1\leqslant m \leqslant Q$ satisfies the restrictions in (\ref{comditions}).
		Then
		
		\begin{equation*}
			\begin{split}
				\sum_{P\in\mathbb{P}_{2g+1}} &\sum\nolimits^{n,0} f\left(\gamma_{j_1,P},\cdots,\gamma_{j_n,P}\right) \\
				& =\frac{1}{\left(2\pi\right)^n} \sum_{\substack{K\cup L\cup M \\=\{1,\cdots,n\}}} \left(-1\right)^{|M|} \left(\int_0^{\pi /\ln q}\right)^n J_{DL}^*\left(-iz_K\cup iz_L,iz_M\right) \\
				& \ \ \ \times f\left(z_1,\cdots,z_n\right) dz_1\cdots dz_n + o\left(|P|\right).
			\end{split}
		\end{equation*}
	\end{theorem}
	
	We now compute the one-level density for the family of prime characters using Theorem \ref{n-density theoremP}.
	
	\begin{equation*}
		\begin{split}
			\int_{-\pi/\log q}^{\pi/\log q} & f(z)   \left(J_{DL}^*(iz,\emptyset)+ J_{DL}^*(-iz,\emptyset)- J^*_{DL}(\emptyset,-iz\right) dz \\
			&= \int_{-\pi/\log q}^{\pi/\log q} f(z)  \sum_{P\in\p} \Bigg[\log|P|- \frac{\xx'\left(\tfrac{1}{2}-iz\right)}{\xx\left(\tfrac{1}{2}-iz\right)}\\
			& \ \ \ + 2 \left(\frac{\z'}{\z}(1+2iz) - \log q |P|^{-iz} \xx\left(\tfrac{1}{2}+iz\right) \z(1-2iz) \right)   \Bigg]dz.
		\end{split}
	\end{equation*}
	Hence, we can see that the above result agrees with the Andrade, Jung and Shamesaldeen result in (\ref{a.j.sh.1-level}).


	\section{Ratios of Quadratic Dirichlet $L$-functions}\label{The Ratios Conjecture}

	Let $K,Q$ be two positive integers, $\alpha_k$ with $1\le k\le K,$ and $\gamma_m$ with $1 \le m \le Q,$ be complex numbers with positive real part. The aim of this section is to rewrite the results of Andrade and Keating \cite{a&kConInMo} as presented in conjecture \ref{Andrade&Keating Ratios Conjecture} into appropriate set notation. As in \cite{a&kConInMo}, the method below is conjectured to give an asymptotic formula for the ratios of $L$-functions averaged over all square-free polynomials $D\in\h$ with an error term $o\left(|D|\right)$ under the restrictions imposed by (\ref{comditions}).
	
	Let $A=\left\{\alpha_1,\cdots,\alpha_K\right\}$ and $B=\left\{\gamma_1,\cdots,\gamma_Q\right\}$ be two finite sets of complex numbers that satisfies the restrictions of (\ref{comditions}). Consider the ratios of Dirichlet $L$-functions
	
	\begin{equation*}
		\begin{split}
			R_{DL}(A,B) &= \sum_{D\in\h} \frac{\prod_{k=1}^K L\left(\frac{1}{2}+\alpha_k,\x_D\right)}{\prod_{m=1}^Q L\left(\frac{1}{2}+\gamma_m,\x_D\right)}.
		\end{split}
	\end{equation*}
	
	We replace the $L\left(\frac{1}{2}+\alpha_k,\x_D\right)$ in the numerator with the ``approximate'' functional equation
	
	\begin{equation*}
		L\left(\tfrac{1}{2}+\alpha_k,\x_D\right) = \sum_{\substack{f_k \in \mm_g}} \frac{\x_D(f_k)}{|f_k|^{\frac{1}{2}+\alpha_k}} + \xx_D\left(\tfrac{1}{2}+\alpha_k\right) \sum_{\substack{f_k \in\mm_{g-1}}} \frac{\x_D(f_k)}{|f_k|^{\frac{1}{2}-\alpha_k}},
	\end{equation*}
	and expand the $L\left(\tfrac{1}{2} +\gamma_m,\x_D\right)$ in the denominator into series,
	
	\begin{equation*}
		\frac{1}{L\left(\frac{1}{2}+\gamma_m,\x_D\right)} = \sum_{h_m \in\mm} \frac{\mu(h_m) \x_D(h_m)}{|h_m|^{\frac{1}{2}+\gamma_m}}
	\end{equation*}
	where $\mu(f)$ and $\x_D(f)$ are defined in Section \ref{background}. Combining the two terms together we have
	
	\begin{equation*}
		\begin{split}
			R_{DL}(A,B) &\approx \sum_{D\in\h} \sum_{\varepsilon_k\in\{-1,1\}^K} |D|^{\sum_k\frac{\varepsilon_k-1}{2}\alpha_k} \prod_{k=1}^K \xx\left(\tfrac{1}{2}+\tfrac{1-\varepsilon_k}{2}\alpha_k\right) \\
			& \ \ \  \times \prod_{k=1}^K \sum_{f_k \in\mm } \frac{\x_D(f_k)}{|f_k|^{\frac{1}{2}+\varepsilon_k\alpha_k}} \prod_{m=1}^Q \sum_{h_m \in\mm} \frac{\mu(h_m) \x_D(h_m)}{|h_m|^{\frac{1}{2}+\gamma_m}}.
		\end{split}
	\end{equation*}
	By computing the expected value of the sum using that
	
	\begin{equation*}
		\begin{split}
			\lim_{g\to\infty} \frac{1}{\#\h} \sum_{D\in\h} \x_D(f) = \begin{cases}
				\underset{\substack{P \text{ prime} \\ P\mid f}}{\prod}\left(1+\frac{1}{|P|}\right)^{-1} & \text{ if } f \ \text{is the square of a polynomial} \\
				0 & \text{ otherwise,}
			\end{cases}
		\end{split}
	\end{equation*}
	we can express $R_{DL}(A,B)$ as
	
	\begin{equation*}
		\begin{split}
			R_{DL}(A,B) &\approx \sum_{D\in\h} \sum_{\varepsilon_k\in\{-1,1\}^K} |D|^{\sum_k\frac{\varepsilon_k-1}{2} \alpha_k} \prod_{k=1}^K \xx\left(\tfrac{1}{2}+\tfrac{1-\varepsilon_k}{2}\alpha_k\right) \\
			& \ \ \ \ \ \ \times \sum_{f_1\cdots f_Kh_1\cdots h_Q=\square} \frac{a\left(f_1\cdots f_Kh_1\cdots h_Q\right)\mu(h_1) \cdots \mu(h_Q)}{\prod_k |f_k|^{\frac{1}{2}+\varepsilon_k\alpha_k}\prod_m |h_m|^{\frac{1}{2}+\gamma_m}} ,
		\end{split}
	\end{equation*}
	where $a(f)=\underset{\substack{P \text{ prime} \\ P\mid f}}{\prod}\left(1+\frac{1}{|P|}\right)^{-1}$. Note that, provided that the real parts of $\alpha_k,\gamma_m$ are sufficiently small, we can rewrite $R_{DL}(A,B)$ in terms of Euler products. Thus,
	
	\begin{equation*}
		\begin{split}
			&R_{DL}(A,B) \approx \sum_{D\in\h} \sum_{\varepsilon_k\in\{-1,1\}^K} |D|^{\sum_k\frac{\varepsilon_k-1}{2} \alpha_k} \prod_{k=1}^K \xx\left(\tfrac{1}{2}+\tfrac{1-\varepsilon_k}{2}\alpha_k\right) \\
			& \ \ \ \times \prod_{\substack{P \text{ monic} \\ \text{irreducible}}} \left(1+\left(1+\frac{1}{|P|}\right)^{-1}\sum_{\underset{k}{\sum}a_k+\underset{m}{\sum}c_m \text{ even}} \frac{\prod_{m=1}^Q \mu\left(P^{c_m}\right)}{|P|^{\underset{k}{\sum}a_k(\frac{1}{2}+\varepsilon_k\alpha_k)+\underset{m}{\sum}c_m(\frac{1}{2}+\gamma_m)}}\right),
		\end{split}
	\end{equation*}
	where
	\begin{equation*}
		\mu\left(P^c\right) = \begin{cases}
			1 & c=0\\
			-1 & c=1\\
			0 & \text{ otherwise}.
		\end{cases}
	\end{equation*}
	The M\"{o}bius function let us divide the set $B=\{\gamma_1,\cdots,\gamma_Q\}$ into two sets $F$ and $G$, where $F$ is the set of all index $m\in\{1,\cdots,Q\}$ that represents the $\gamma_m$'s associated to $c_m=0$ and $G$ is the set of all index $m\in\{1,\cdots,Q\}$ that represents the $\gamma_m$'s associated to $c_m=1$, in such a way that $F\cup G=\{1,\cdots,Q\}$. With this in hand, we can write
	
	\begin{equation}\label{co}
		\begin{split}
			&R_{DL}(A,B) \approx \sum_{D\in\h} \sum_{\varepsilon_k\in\{-1,1\}^K} |D|^{\sum_k\frac{\varepsilon_k-1}{2} \alpha_k} \prod_{k=1}^K \xx\left(\tfrac{1}{2}+\tfrac{1-\varepsilon_k}{2}\alpha_k\right) \\
			&  \times \prod_{\substack{P \text{ monic} \\ \text{irreducible}}} \left(1+\left(1+\frac{1}{|P|}\right)^{-1}\sum_{\substack{F\cup G\\=\{1,\cdots,Q\}}}\sum_{\substack{\underset{k}{\sum}a_k+\#G \\\text{ even}}} \frac{\left(-1\right)^{\#G}}{|P|^{\underset{k}{\sum} (\frac{1}{2}+\varepsilon_k\alpha_k)a_k + \underset{m\in G}{\sum} (\frac{1}{2}+\gamma_m)}}\right).
		\end{split}
	\end{equation}
	
	Now, we want to identify which terms in (\ref{co}) are divergent for small $\RR(\alpha_k)$ and $\RR(\gamma_m)$. Thus, recalling the restriction in (\ref{comditions}), $|\RR(\alpha_j)|,|\RR(\gamma_j)|<\frac{1}{4}$, we can see that the terms of the form $\frac{1}{|P|^{2-2\alpha_j}}$ and higher powers are convergent. Therefore, the only terms with indices such that $\sum_ka_k+\#G\leqslant 2$ are divergent. Now, we have that $\sum_ka_k+\#G$ is even, thus the only values are when it is $0$ and $2$. However, we have already pulled out the term when it is $0$ in (\ref{co}). Hence we only need to consider the terms such that $\sum_k a_k+\#G=2$. In the end, we have that
	
	\begin{equation*}
		\begin{split}
			R_{DL}(A,B) &\approx \sum_{D\in\p} \sum_{\varepsilon_k\in\{-1,1\}^K} |D|^{\sum_k\frac{\varepsilon_k-1}{2} \alpha_k} \prod_{k=1}^K \xx\left(\tfrac{1}{2}+\tfrac{1-\varepsilon_k}{2}\alpha_k\right) \\
			& \ \ \ \times \prod_{\substack{P \text{ monic} \\ \text{irreducible}}} \Bigg(1+ \left(1+\frac{1}{|P|}\right)^{-1} \Bigg(\sum_k \frac{1}{|P|^{1+2\varepsilon_k\alpha_k}} + \sum_{\substack{k,j\\k<j}} \frac{1}{|P|^{1+\varepsilon_k\alpha_k+\varepsilon_l\alpha_l}} \\
			&\ \ \  +\sum_{k,m} \frac{-1}{|P|^{1+\varepsilon_k\alpha_k+\gamma_m}} + \sum_{\substack{m,r\\m<r}} \frac{1}{|P|^{1+ \gamma_m+\gamma_r}}  + \cdots\Bigg)\Bigg),
		\end{split}
	\end{equation*}
	where $\cdots$ indicates terms that converge. 
	
	As in Andrade and Keating's paper \cite{a&kConInMo}, we can write the contribution of these non-convergent terms in the form of products of zeta functions associated with $\A=\f$. Recall that,
	
	\begin{equation}\label{Zz1}
		\begin{split}
			\varepsilon A &= \{\varepsilon_1\alpha_1,\cdots,\varepsilon_K\alpha_K\},\\
			Z_{\z}(A,B)&= \prod_{\substack{\alpha\in A \\ \beta\in B}} \z(1+\alpha+\beta),\\
			Y_{\z}(A) &= \prod_{\substack{\alpha\in A}} \z(1+2\beta),\\
		\end{split}
	\end{equation}
	so, using the above we have that
	
	\begin{equation*}
		\begin{split}
			&\frac{\underset{j\leqslant k\leqslant K}{\prod} \z\left(1+\varepsilon_j\alpha_j+\varepsilon_k\alpha_k\right) \underset{m< r\leqslant Q}{\prod} \z\left(1+\gamma_m+\gamma_r\right)}{\prod_{k=1}^K\prod_{m=1}^Q \z\left(1+\varepsilon_k\alpha_k+\gamma_m\right)}\\
			&\; \; \; \; \;\;\;\;\;\;\;\;\;\;\;\;\;\;\;\; \; \; \; \;\;\;\;\;\;\; \; \; =\sqrt{\frac{Z_{\z}\left(\varepsilon A,\varepsilon A\right) Y_{\z}\left(\varepsilon A\right) Z_{\z}\left(B,B\right)}{Z_{\z}\left(\varepsilon A,B\right)^2 Y_{\z}\left(B\right)}}.
		\end{split}
	\end{equation*}
	So we can write
	
	\begin{equation*}
		\begin{split}
			R_{DL} (A,B) &= \sum_{D\in\h} \sum_{\varepsilon\in\{-1,1\}^k} |D|^{\sum_k \frac{1-\varepsilon_k}{2}\alpha_k} \prod_{k=1}^K \xx\left(\tfrac{1}{2}+\tfrac{1-\varepsilon_k}{2}\alpha_k\right)\\
			& \ \ \ \times Y_{DL}(A,B,\varepsilon) A_{DL}(A,B,\varepsilon),
		\end{split}
	\end{equation*}
	with
	
	\begin{equation*}
		\begin{split}
			Y_{DL}(A,B,\varepsilon) &= \sqrt{\frac{Z_{\z}(\varepsilon A,\varepsilon A) Y_{\z}(\varepsilon A) Z_{\z}(B,B)}{Z_{\z}(\varepsilon A,B)^2 Y_{\z}(B)}},\\
		\end{split}
	\end{equation*}
	\begin{equation*}
		\begin{split}
			&A_{DL}(A,B,\varepsilon) = Y_{DL}(A,B,\varepsilon)^{-1} \\
			& \times \prod_{\substack{P \text{ monic} \\ \text{irreducible}}} \left(1+\left(1+\frac{1}{|P|}\right)^{-1}\sum_{\substack{F\cup G\\=\{1,\cdots,Q\}}}\sum_{\substack{\underset{k}{\sum}a_k+\#G \\\text{ even}}} \frac{\left(-1\right)^{\#G}}{|P|^{\underset{k}{\sum} (\frac{1}{2}+\varepsilon_k\alpha_k)a_k + \underset{m\in G}{\sum} (\frac{1}{2}+\gamma_m)}}\right).
		\end{split}
	\end{equation*}{\color{white}'}
	
	Recall the restrictions in (\ref{comditions}), the bound of $\frac{1}{4}$ on the absolute value of the real parts of the shifts are to prevent divergence of the Euler products that appear in the ratios conjecture, see Conrey and Snaith section 2 in \cite{Conrey&Snaith} . Therefore, the Euler product $A_{DL}(A,B,\varepsilon)$ is convergent for all of the variables in small disks around $0$.
	
	It remains to write $A_{DL}(A,B,\varepsilon)$ into a closed form expression. Remind that,
	
	\begin{equation}\label{V_+1}
		V_{+}(A) = \prod_{\alpha\in A} \left(1+\frac{1}{|P|^{\frac{1}{2}+\alpha}}\right),\ \text{and} \ \ \ V_{-}(A) = \prod_{\alpha\in A} \left(1-\frac{1}{|P|^{\frac{1}{2}+\alpha}}\right).\\
	\end{equation}
	Therefore we can write $A_{DL}(A,B,\D)$, where $\D\subseteq A,$ as
	
	\begin{equation}\label{A_{DL}1}
		\begin{split}
			&A_{DL}(A,B,\D) = Y_{DL}(A,B,\D)^{-1} \\
			&  \times \prod_{\substack{P \text{ monic} \\ \text{irreducible}}} \left(1+\frac{1}{|P|}\right)^{-1}  \left(\frac{V_{-}(B)}{2V_{-}((A\setminus \D) \cup \D^{-})} +  \frac{V_{+}(B)}{2V_{+}((A\setminus \D) \cup \D^{-})}+\frac{1}{|P|}\right),
		\end{split}
	\end{equation}
	where
	
	\begin{equation}\label{Y_{DL}1}
		\begin{split}
			Y_{DL}&(A,B,\D) \\
			&= \sqrt{\frac{Z_{\z}((A\setminus \D) \cup \D^{-},(A\setminus \D) \cup \D^{-})  Y_{\z}((A\setminus \D) \cup \D^{-}) Z_{\z}(B,B)}{Z_{\z}((A\setminus \D) \cup \D^{-},B)^2 Y_{\z}(B)}}.
		\end{split}
	\end{equation}
	
	Finally, combining all of the above, we can rewrite Conjecture \ref{Andrade&Keating Ratios Conjecture} into a set form as given in Conjecture \ref{Our Ratios Conjecture2.}.
	
	%
	
	
	\section{Differentiating Ratios of Quadratic Dirichlet $L$-functions}\label{DifRatios}
	
	In this section we compute a formula for
	
	\begin{equation}\label{dR_{A,B}}
		\begin{split}
			J_{DL} (A) &= \sum_{D\in \h} \prod_{\alpha\in A} \frac{L'}{L} \left(\frac{1}{2}+\alpha,\x_D\right) \\
			&= \prod_{\alpha\in A} \frac{\partial}{\partial\alpha} R_{DL}(A,B)\Big\vert_{B=A}.
		\end{split}
	\end{equation}{\color{white}'}
	In other words, we prove theorem \ref{J^*_{DL}.}.
	
	The main steps to compute the above are by pulling the differentiation inside the sum over subsets $\D\subseteq A$ in equation (\ref{df}) and then separating the differentiations in cases either $\alpha\in \D$ or $\alpha \notin \D$. Notice that the differentiation by $\alpha\in \D$ is straightforward. However, to complete the calculations, we need to use the logarithmic differentiation method, which is given in Lemma \ref{H} below. 
	
	We can see from equation (\ref{dR_{A,B}}) that the size of the sets $A$ and $B$ must be equal, therefore let $A=\{\alpha_1,\cdots,\alpha_K\}$ and $B=\{\beta_1,\cdots,\beta_K\}$ where $\alpha_i$ and $\beta_i$ have positive real part. For the substitution, we can think of it as substituting $\alpha_i$ for $\beta_i$ for each $1\leqslant i\leqslant K$. Recalling that we can write $A=\D\cup (A\backslash \D)$, therefore the summation is running over all the possible partitions of $A$ which is partitioned into two sets, including $A=\emptyset \cup A$. Note that we are considering separately differentiating by the variables within $\D$ and the other variables within $A\backslash \D$. 
	
	We start by expanding the term within the square root in equation (\ref{Y_{DL}1}) by letting $B=B_\D\cup B_{(A\backslash \D)}$ where $B_\D$ is the set that will eventually be substituted by $\D$.
	
	From equation (\ref{Zz1}) we can easily see that $Z_{\z}(A,B)=Z_{\z}(B,A)$ and $Z_{\z}(A\cup B,C)=Z_{\z}(A,C)Z_{\z}(B,C)$ when $A,B$ are disjoint. These let us write
	
	\begin{equation*}
		\begin{split}
			&\frac{Z_{\z}((A\setminus \D) \cup \D^{-},(A\setminus \D) \cup \D^{-})  Y_{\z}((A\setminus \D) \cup \D^{-}) Z_{\z}(B,B)}{Z_{\z}((A\setminus \D) \cup \D^{-},B)^2 Y_{\z}(B)} \\
			&\ = Z_{\z}((A\setminus \D) ,(A\setminus \D)) Z_{\z}((A\setminus \D),\D^{-}) Z_{\z}(\D^{-},(A\setminus \D)) Z_{\z}(\D^{-},\D^{-})\\
			&\ \ \ \times \frac{Z_{\z}(B_\D,B_\D) Z_{\z}(B_\D, B_{A\backslash \D}) Z_{\z}( B_{A\backslash \D},B_\D) Z_{\z}( B_{A\backslash \D}, B_{A\backslash \D}) }{ Z_{\z}((A\setminus \D) ,B_\D)^2 Z_{\z}((A\setminus \D) , B_{A\backslash \D})^2 Z_{\z}( \D^{-},B_\D)^2 Z_{\z}( \D^{-}, B_{A\backslash \D})^2}\\
			& \ \ \ \times \frac{Y_{\z}((A\setminus \D))  Y_{\z}(\D^{-})}{ Y_{\z}(B_\D)  Y_{\z}(B_{A\backslash \D})}.
		\end{split}
	\end{equation*} {\color{white}'}
	
	Making the substitution $B=A,$ we get
	
	\begin{equation*}
		\begin{split}
			&\frac{Z_{\z}((A\setminus \D) \cup \D^{-},(A\setminus \D) \cup \D^{-})  Y_{\z}((A\setminus \D) \cup \D^{-}) Z_{\z}(B,B)}{Z_{\z}((A\setminus \D) \cup \D^{-},B)^2 Y_{\z}(B)} \\
			&= \frac{Z_{\z}(\D^{-},\D^{-}) Z_{\z}(\D,\D) Y_{\z}(\D^{-})}{Z_{\z}(\D^{-},D)^2 Y_{\z}(\D)},
		\end{split}
	\end{equation*}
	which is zero unless $\D$ is empty since we have $Z_{\z}(\D^{-}, \D)$ in the denominator and $\z(s)$ has a pole at $s=1$.
	
	When we differentiate (\ref{dR_{A,B}})  by a given $\alpha\in \D$, the expression $\z(\beta-\alpha),$ which comes from $Z_{\z}(\D^{-},B_\D)$ term in the denominator, must be differentiated. Otherwise, when the substitution occurs the whole expression is equal to zero. So when differentiating $\frac{1}{Z_{\z}(\D^{-},B_\D)}$ for every $\alpha\in \D$ the only term that does not become zero on substitution is the one where each $\frac{1}{\z(1+\beta_\alpha-\alpha)}$ is differentiated by $\alpha$. Here $\beta_\alpha$ is the variable that will be substituted by $\alpha$ when $A$ is substituted for $B$.
	
	Note that
	
	\begin{equation*}
		\begin{split}
			\frac{d}{d\alpha} \frac{1}{\z(1+\beta-\alpha)} 
			&= \frac{d}{d\alpha} \left(1-q^{\alpha-\beta}\right)\\
			&=- q^{\alpha-\beta}\log q.
		\end{split}
	\end{equation*}
	Therefore, it is easy to check that
	
	\begin{equation*}
		\begin{split}
			\prod_{\alpha\in \D} \frac{d}{d\alpha}\frac{1}{Z_{\z}(\D^{-},B_\D)}\Big\vert_{B_\D=\D}   = \frac{( -\log q)^{\#\D}}{Z^{\dagger}(\D^{-},\D)} ,
		\end{split}
	\end{equation*}
	where
	\begin{equation*}
		Z^{\dagger}(A,B) = \prod_{\substack{\alpha\in A \\ \beta\in B\\\alpha+\beta\neq0}} \z(1+\alpha+\beta),
	\end{equation*}
	and the $\dagger$ indicates that a factor is omitted if its argument is zero $(\alpha +\beta =0)$, since we only include the terms of $Z$ that are not equal to $\z(1)$. Hence
	
	\begin{equation*}
		\begin{split}
			&J_{DL}(A) \\
			&= \sum_{D\in\h} \sum_{\D\subseteq A} \left|D\right|^{- \sum_{\delta\in \D}\delta } \prod_{\delta\in \D} \xx\left(\tfrac{1}{2}+\delta\right)   \left(- \log q\right)^{\#\D} \sqrt{\frac{ Z_{\z}(\D^{-},\D^{-}) Z_{\z}(\D,\D) Y_{\z}(\D^{-})  }{  Z_{\z}^\dag( \D^{-},\D)^2  Y_{\z}(\D)}} \\
			& \ \ \times   \prod_{\alpha\in A\setminus \D}\frac{\partial}{\partial\alpha} \Bigg( \sqrt{\frac{Z_{\z}(A\setminus \D,A\setminus \D) Z_{\z}(A\setminus \D,\D^{-})^2 Z_{\z}(B_\D, B_{A\backslash \D})^2 }{Z_{\z}(A\setminus \D ,B_\D)^2 Z_{\z}(A\setminus \D , B_{A\backslash \D})^2 Z_{\z}( \D^{-}, B_{A\backslash \D})^2 }}\\
			& \ \ \ \times \sqrt{\frac{ Z_{\z}( B_{A\backslash \D}, B_{A\backslash \D}) Y_{\z}(A\setminus \D)}{ Y_{\z}(B_{A\backslash \D})}}  A_{DL}(A,B,\D)\Bigg)\Bigg\vert_{B=A}+ o\left(|D|\right).
		\end{split}
	\end{equation*}
	
	We now use the following lemma, which is quoted from Mason's PhD thesis \cite{mason1}, to carry out the remaining differentiations.

	\begin{lemma}[Lemma 4.4, \cite{mason1}]\label{H}
		Let $H$ be a differentiable function of $w\in W$. Then
		\begin{equation*}
			\left(\prod_{w\in W} \frac{\partial}{\partial w}\right) e^H = e^H \sum_{W=W_1+\cdots + W_r} H(W_1)\cdots H(W_r)
		\end{equation*}
		where
		
		\begin{equation*}
			H(W)=\left(\prod_{w\in W} \frac{\partial}{\partial w}\right) H,
		\end{equation*}
		the sum over $W_i$ is a sum over all distinct set partitions of $A\setminus \D$.
	\end{lemma}
	
	Let
	
	\begin{equation}\label{H^{A_{DL},B}_D}
		\begin{split}
			H & = \log\left(H_\D^{A,B}\right) \\
			&= \sum_{\substack{\alpha\in A\setminus \D \\ \beta\in \D}} \log\left(\z(1+\alpha-\beta)\right) + \frac{1}{2} \sum_{\substack{\alpha\in A\setminus \D \\ \beta\in A\setminus \D}} \log\left(\z(1+\alpha+\beta)\right)\\
			&\ \ \ + \frac{1}{2} \sum_{\substack{\alpha\in A\setminus \D }} \log\left(\z(1+2\alpha)\right)  - \sum_{\substack{\alpha\in A\setminus \D \\ \beta\in B_\D}} \log\left(\z(1+\alpha+\beta)\right)\\
			&\ \ \   - \sum_{\substack{\alpha\in A\setminus \D \\ \beta\in B_{A\backslash \D}}} \log\left(\z(1+\alpha+\beta)\right) + \log\left( A_{DL}(A,B,\D)\right).
		\end{split}
	\end{equation}
	Then, by applying Lemma \ref{H} we have
	
	\begin{equation*}
		\begin{split}
			& \sqrt{\frac{Z_{\z}(B_\D,B_{A\setminus \D})^2 Z_{\z}(B_{A\setminus \D},B_{A\setminus \D})} {Z_{\z}(\D^{-},B_{A\setminus \D})^2 Y_{\z}(B_{A\setminus \D})}}  \prod_{\alpha \in A\setminus \D} \frac{\partial}{\partial\alpha}\left(e^{H}\right) \Bigg\vert_{\substack{B_{A \setminus \D = A \setminus \D} \\ B_\D=\D}} \\
			&= \sqrt{\frac{Z_{\z}(B_\D,B_{A\setminus \D})^2 Z_{\z}(B_{A\setminus \D},B_{A\setminus \D})} {Z_{\z}(\D^{-},B_{A\setminus \D})^2 Y_{\z}(B_{A\setminus \D})}} \ H_\D^{A,B}  \sum_{A\setminus \D=W_1+\cdots+W_R} \prod_{\alpha \in W_r} H(W_r) \Bigg\vert_{\substack{B_{A \setminus \D = A \setminus \D} \\ B_\D=\D}} \\
			&= A_{DL}(A,B,\D) \sum_{A\setminus \D=W_1+\cdots+W_R} \widetilde{H}_\D^{A,B}(W_r) \Bigg\vert_{\substack{B= A}}.
		\end{split}
	\end{equation*}
	By letting $B_{A\setminus \D}=A\setminus \D$ and $B_\D=\D $ we have
	
	\begin{equation}\label{HB=D}
		\begin{split}
			&\Bigg(\sqrt{\frac{Z_{\z}(B_\D,B_{A\setminus \D})^2 Z_{\z}(B_{A\setminus \D},B_{A\setminus \D})} {Z_{\z}(\D^{-},B_{A\setminus \D})^2 Y_{\z}(B_{A\setminus \D})}} \ H_\D^{A,B}   \Bigg)\Bigg\vert_{\substack{B_{A \setminus \D = A \setminus \D} \\ B_\D=\D}} \\
			&\\
			& =  \left(A_{DL}(A,B,\D)\right)\Bigg\vert_{\substack{B=A}} \sqrt{\frac{Z_{\z}(D,A\setminus \D)^2 Z_{\z}(A\setminus \D,A\setminus \D)} {Z_{\z}(\D^{-},A\setminus \D)^2 Y_{\z}(A\setminus \D)}} \\
			& \ \ \ \times \sqrt{\frac{Z_{\z}(A\setminus \D,\D^{-})^2  Z_{\z}(A\setminus \D,A\setminus \D) Y_{\z}(A\setminus \D)}{ Z_{\z}(A\setminus \D ,\D)^2 Z_{\z}(A\setminus \D , A\backslash \D)^2  }}   \\
			&= \left(A_{DL}(A,B,\D)\right)\Bigg\vert_{\substack{B=A}} .
		\end{split}
	\end{equation}
	Before we are in a position to complete the proof of the theorem we need to compute these $\widetilde{H}_\D^{A,B}(w_r)$ terms and $A_{DL}(A,B,\D)\big\vert_{B=A}$. These are computed below.
	
	\begin{lemma}\label{A_{DL}(A,b,D)}
		Let $A_{DL}(A,B,\D)$ be defined as in (\ref{A_{DL}}), then we have
		
		\begin{equation*}
			A_{DL}(A,B,\D)\big\vert_{B=A}=A_{DL}(\D,\D,\D),
		\end{equation*}
		and this is an analytic function.
	\end{lemma}
	\begin{proof}
	Substituting the set $B$ with the set $A$ and using the fact that $Z(A\cup C,B)=Z(A,B)Z(C,B)$ , $V_{\pm}(A\cup C)=V_{\pm}(A) V_{\pm}(C)$ whenever $A$ and $C$ are distinct sets, and that $Z(A,B)=Z(B,A)$, we have
		\begin{equation*}
			\begin{split}
				&A_{DL}(A,B,\D)\Big\vert_{B=A} \\
				&=  \sqrt{\frac{Z_{\z}(A\setminus \D \cup \D^{-},\D\cup A\setminus \D)^2 Y_{\z}(\D\cup A\setminus \D)}{Z_{\z}((A\setminus \D) \cup \D^{-},(A\setminus \D) \cup \D^{-})  Y_{\z}((A\setminus \D) \cup \D^{-}) Z_{\z}(\D\cup A\setminus \D,\D\cup A\setminus \D)}}\\
				& \ \ \ \times \prod_{\substack{P \text{ prime}}} \left(1+\frac{1}{|P|}\right)^{-1}\left( \frac{V_{-}(\D) V_{-}( A\setminus \D)}{2V_{-}(A\setminus \D) V_{-} (\D^{-})} +  \frac{V_{+}(\D) V_{+}(A\setminus \D)}{2V_{+}(A\setminus \D) V_{+}( \D^{-})}+\frac{1}{|P|}\right)\\
				&=  \sqrt{\frac{Z_{\z}(\D^{-} ,\D)^2 Y_{\z}(\D)}{Z_{\z}(\D^{-},\D^{-})  Y_{\z}(\D^{-}) Z_{\z}(\D, \D)}} \prod_{\substack{P \text{ prime}}} \left(1+\frac{1}{|P|}\right)^{-1}\left( \frac{V_{-}(\D)}{2 V_{-} (\D^{-})} +  \frac{V_{+}(\D)}{2 V_{+}( \D^{-})}+\frac{1}{|P|}\right)\\
				&= A_{DL}(\D,\D,\D),
			\end{split}
		\end{equation*}
		which is an analytic function.
	\end{proof}
	
	Now from equation (\ref{H^{A_{DL},B}_D}), we can easily show that, for $W_r\subset A\setminus \D$,
	
	\begin{equation}\label{widetilde{H}^{A,B}_D}
		\begin{split}
			\widetilde{H}^{A,B}_\D (W_r)&= \prod_{\alpha\in W_r}\frac{\partial}{\partial \alpha} H^{A,B}_\D\\
			&= H^{A,B}_\D(W_r)+A^{A,B}_\D(W_r)
		\end{split}
	\end{equation}
	with
	
	\begin{equation}\label{{H}^{A,B}_D}
		\begin{split}
			H^{A,B}_\D(W_r) &= \begin{cases}
				\underset{\beta\in D}{\sum} \left(\frac{\z'}{\z}(1+\gamma-\beta) - \frac{\z'}{\z}(1+\gamma+\beta)\right)+\frac{\z'}{\z}(1+2\gamma), & W_r=\{\gamma\}\\
				\left(\frac{\z'}{\z}\right)' (1+\gamma_1+\gamma_2), & W_r=\{\gamma_1,\gamma_2\}\\
				0, & |W_r|\geqslant 3
			\end{cases}
		\end{split}
	\end{equation}
	and therefore
	
	\begin{equation*}
		A^{A,B}_\D(W_r) = \prod_{\alpha\in W_r} \frac{\partial}{\partial \alpha} \log A_{DL} (A,B,\D) \Bigg\vert_{B=A}.
	\end{equation*}
	
	Hence, Theorem \ref{J^*_{DL}.} follows.
	%
	%
	%
	
	
	\section{Residue Theorem}\label{Residue Theorem}
	
	In this section we locate the poles of $J^*_{DL}(A)$ and compute the residue at these poles. Recalling theorem \ref{J^*_{DL}.}, we have that
	
	\begin{equation}\label{J^*_{DL}1}
		\begin{split}
			J^*_{DL}(A) = & \sum_{D\in\h} \sum_{\D\subseteq A} \left|D\right|^{- \sum_{\delta\in \D}\delta } \prod_{\delta\in \D} \xx\left(\tfrac{1}{2}+\delta\right)  \left( -\log q\right)^{\#D}\\
			& \times  \sqrt{\frac{	 Z_{\z}(\D^{-},\D^{-}) Z_{\z}(\D,\D) Y_{\z}(\D^{-})  }{  Z_{\z}^\dag( \D^{-},D)^2  Y_{\z}(\D)}} \\
			& \ \ \ \times A_{DL}(\D,\D,\D) \sum_{A\setminus \D = W_1+\cdots+W_R} \prod_{r=1}^R \widetilde{H}^{A,B}_\D(W_r).
		\end{split}
	\end{equation}
Since $\z(s)$ has a pole when $s=1$, the only possible poles for $J^*_{DL}(A)$ are coming either from $Y_{\z}(A)$ or $Z_{\z}(A,A)$ and that happens when $\alpha=0$ or $\alpha=-\beta$ where $\alpha,\beta\in A$. We will compute the residues at these poles. In the following lemmas, we write the series expansions around these poles for $A^{A,B}_\D(W_r), H^{A,B}_\D(W_r)$ and $A_{DL}(\D,\D,\D)$. Note that the proofs are omitted since they follow from similar calculations carried out in number field setting in \cite{mason1} and \cite{mason&Snaith}. For the function field calculations, similar lemmas were proved in \cite{shams}.
	
	\begin{lemma}\label{expansion A_D^{A,B}}
		Let $\alpha^*,\beta^*\in A,$ and define $\D'= \D \cap (A/\{\alpha^*,\beta^*\}).$ Then about $\alpha^*=-\beta^*$, we have
		
		\begin{equation*}
			\begin{split}
				A^{A,B}_\D(W_r) &= A^{A,B}_{\D'}(W_r) -(\alpha^*+\beta^*) \Big( A^{A,B}_{\D'}\left(W_r+\{\beta^*\}\right) \\
				& \ \ \ +A^{A,B}_{\D'}(W_r+\{-\beta^*\}) \Big) +O\left(\left(\alpha^*+\beta^*\right)^2\right).
			\end{split}
		\end{equation*}
	\end{lemma}
	
	\begin{lemma}\label{expansion H_D^{A,B}}
		Let $\alpha^*,\beta^*\in A, W_r\subset A\backslash \D$ and $\D'=\D\cap (A\backslash \{\alpha^*,\beta^*\})$. Then about $\alpha^*=-\beta^*$, we have
		
		\begin{equation*}
			\begin{split}
				H_\D^{A,B}(W_r) &= H_{\D'}^{A,B}(W_r) - (\alpha^*+\beta^*) \Big(H_{\D'}^{A,B}(W_r+\{\beta^*\}) \\
				&\ \ \ + H_{\D'}^{A,B}(W_r+\{-\beta^*\}) \Big) +O\left((\alpha^*+\beta^*)^2\right).
			\end{split}
		\end{equation*}
	\end{lemma}
	
	\begin{lemma}\label{expansion A_{DL}}
		Let $\alpha^*,\beta^*\in A$ and $\D'=\D\cap (A\backslash \{\alpha^*,\beta^*\})$. Then about $\alpha^*=-\beta^*$
		
		\begin{equation*}
			\begin{split}
				A_{DL}(\D,\D,\D) = A_{DL}(\D',\D',\D')& \Big[ 1- (\alpha^*+\beta^*) \big(A_{\D'}^{A,B}(\{\beta^*\})\\
				&+ A_{\D'}^{A,B}(\{-\beta^*\})\big) + O\left((\alpha^*+\beta^*)^2\right)\Big],
			\end{split}
		\end{equation*}
		where $A_{DL}(A,B,\D)$ is defined in (\ref{A_{DL}1}).
	\end{lemma}
	

	\subsection{Residue Theorem for $J^*_{DL}(A)$}\label{Residue Theorem for $J^*_{DL}(A)$}
	
	In this subsection, we use the lemmas above to prove the Residue Theorem for $J_{DL}^*(A)$.
	
	\begin{theorem}\label{Residue}
		Let $A$ be a finite set of complex numbers. Let also $\alpha^*,\beta^*\in A$ and $A'=A\backslash\{\alpha^*,\beta^*\}$ and let $J^*_{DL}(A)$ be as in equation (\ref{J^*_{DL}1}). Then
		
		\begin{equation*}
			\begin{split}
				\underset{\alpha^*=-\beta^*}{\mathrm{Res}}&(J^*_{DL}(A)) \\
				& =J^*_{DL}(A'\cup \{\beta^*\}) + J^*_{DL}(A'\cup \{-\beta^*\}) -\frac{\xx'_{D}}{\xx_{D}} \left(1+\beta^*\right) J^*_{DL}(A')
			\end{split}
		\end{equation*}
		where $\frac{\xx'_{D}}{\xx_{D}}(s)=-\log|D|+\log q$ comes from the functional equation. We also have
		
		\begin{equation*}
			\underset{\alpha^*=0}{\mathrm{Res}}(J^*_{DL}(A))=0.
		\end{equation*}
	\end{theorem}
	\begin{proof}
		Let $\D\subseteq A$. Define $P_\D$ and $Q$ as
		
		\begin{equation*}
			\begin{split}
				Q(\D) &= \left(-\log q\right)^{\#\D} |D|^{-\sum_{\delta\in \D}\delta} \prod_{\delta\in \D} \xx\left(\tfrac{1}{2}+\delta\right) \\
				& \ \ \ \times  \sqrt{\frac{ Z_{\z}(\D^{-},\D^{-}) Z_{\z}(\D,\D) Y_{\z}(\D^{-})  }{  Z_{\z}^\dag( \D^{-},\D)^2  Y_{\z}(\D)}}, \\
				\widetilde{Q}(\D) &= Q(\D)  A_{DL}(\D,\D,\D),\\
			\end{split}
		\end{equation*}
		and
		
		\begin{align*}
			P_\D (A\backslash \D)&=\widetilde{Q}(\D) \sum_{\substack{A\backslash \D=\sum_{r\leqslant R} W_r \\ |W_r|\leqslant 2}} \prod_{r=1}^R \left(\widetilde{H}^{A,B}_\D(W_r)\right)  ,
		\end{align*}
		where $	\widetilde{H}^{A,B}_\D(W_r) = H_\D^{A,B}(W_r)+A_D^{A,B}(W_r)$ is defined as in (\ref{widetilde{H}^{A,B}_D}). Thus, we can write
		
		\begin{equation*}
			\begin{split}
				J^*_{DL}(A)  
				&= \sum_{\D\subseteq A} P_\D(A\backslash \D).
			\end{split}
		\end{equation*}
		In order to find the residue of $J^*_{DL}(A)$ we start by showing that the pole at $\alpha^*=-\beta^*$ is simple and $\mathrm{Res}_{\alpha^*=-\beta^*}(J^*_{DL}(A))= J^*_{DL}(A'\cup \{\beta^*\}) + J^*_{DL}(A'\cup \{-\beta^*\}) -\frac{\mathcal{X}'_{D}}{\mathcal{X}_{D}} \left(1+\beta^*\right) J^*_{DL}(A') $. Let $\alpha^*,\beta^*\in A$, $A'=A\backslash \{\alpha^*,\beta^*\}, \D'=\D\cap A',$ $(A\backslash \D)'=(A\backslash \D)\cap A'$ and consider the residue of $P_\D(A\backslash \D)$ when $\alpha^*\to-\beta^*$. Given the above, we have the following 4 cases:
		
		\begin{enumerate}
			\item $\alpha^*,\beta^*\in A\backslash \D$
			\item $\alpha^*\in \D ,\beta^*\in A\backslash \D$
			\item $\alpha^*\in A\backslash \D,\beta^*\in \D$
			\item $\alpha^*,\beta^*\in \D$
		\end{enumerate}
		
		\textit{\textbf{Case 1:}} Suppose $\alpha^*,\beta^*\in A\backslash \D$, then $\widetilde{Q}(\D)$ is independent of $\alpha^*$ and $\beta^*$. Therefore, as $\alpha^*\to-\beta^*$ the only pole of $P_\D(A\backslash \D)$ comes when $W_r=\{\alpha^*,\beta^*\},$ i.e., when
		
		\begin{equation*}
			H_\D^{A,B}\left(\{\alpha^*,\beta^*\}\right) = \left(\frac{\z'}{\z}\right)' \left(1+\alpha^*+\beta^*\right).
		\end{equation*}
		We know that $A_\D^{A,B}(W_r)$ is analytic for all $W_r$ and that $	\left(\frac{\z'}{\z}\right)'(x) = \frac{q^{1+x} (\log q)^2}{\left(q-q^x\right)^2},$ with the expansion $\left(\frac{\z'}{\z}\right)'(1+x) = \frac{1}{x^2} + O(1).$ So, we have that
		
		\begin{equation*}
			\widetilde{H}^{A,B}_\D(W_r) =\begin{cases}
				\frac{1} {\left(\alpha^*+\beta^*\right)^2} +O(1), & \text{if} \ W_r=\{\alpha^*,\beta^*\},\\
				O(1), & \text{otherwise}.
			\end{cases}
		\end{equation*}
		Thus,
		
		\begin{equation*}
			\begin{split}
				P_\D&(A\backslash \D) 
				= \frac{1} {\left(\alpha^*+\beta^*\right)^2}  \ \widetilde{Q}(\D)   \sum_{\substack{(A\backslash \D)'=\sum W_r \\ |W_r|\leqslant 2 }}\prod_{r} \widetilde{H}^{A,B}_\D(W_r) +O(1)\\
				&= \frac{1} {\left(\alpha^*+\beta^*\right)^2} \ P_\D\left((A\backslash \D)'\right) +O(1).
			\end{split}
		\end{equation*}
		It is clear that $P_\D\left((A\backslash \D)'\right)$ is independent of $\alpha^*$ and $\beta^*$. Therefore, $P_\D(A\backslash \D)$ will not contain the term $1/\left(\alpha^*+\beta^*\right)$. Hence,  $\underset{\alpha^*\to\beta^*}{Res} \left(P_\D\left((A\backslash \D)'\right)\right)=0.$ \\ $\ \ $ \\
		
		
		\textit{\textbf{Case 2:}} Suppose $\alpha^*\in \D, \beta^*\in A\backslash \D.$ Then $Q(\D)$ is independent of $\beta^*$ and regular as $\alpha^*\to-\beta^*$. It can been seen that the only pole of $P_\D(A\backslash \D)$ comes when $W_r=\{\beta^*\}$, i.e., when
		
		\begin{equation*}
			\begin{split}
				H_\D^{A,B}\left( \{\beta^*\}\right) = \sum_{\gamma\in \D} \left(\frac{\z'}{\z}(1+\beta^*-\gamma)-\frac{\z'}{\z}(1+\beta^*+\gamma)\right) -\frac{\z'}{\z}(1+2\beta^*).
			\end{split}
		\end{equation*}
		When $\gamma=\alpha^*$, we can see that $-\frac{\z'}{\z}(1+\beta^*+\gamma)$ has a simple pole with residue of $1$, that is
		
		\begin{equation*}
			\underset{\alpha^*\to-\beta^*}{\text{Res}} \left(\frac{\z'}{\z}(1+\alpha^*+\beta^*)\right)= 	\underset{\alpha^*\to-\beta^*}{\text{Res}} \left( \frac{ \log q}{1-q^{\alpha^*+\beta^*}} \right)
		\end{equation*}
		by the series expansion given by $-\frac{\z'}{\z}(1+x)= \frac{1}{x} +O(1)$. Note that since $A_\D^{A,B}(W_r)$ has no poles, we have
		
		\begin{equation*}
			\widetilde{H}^{A,B}_\D = \begin{cases}
				\frac{1}{\alpha^*+\beta^*}+O(1) & \text{if}, \ W_r=\{\beta^*\},\\
				O(1), & \text{otherwise}.
			\end{cases}
		\end{equation*}
		Therefore,
		
		\begin{equation*}
			\begin{split}
				P_\D(A\backslash \D) 
				&= \frac{1}{\alpha^*+\beta^*} \ \widetilde{Q}(\D) \sum_{\substack{(A\backslash \D)'=\sum W_r \\ |W_r|\leqslant 2 }} \prod_r \widetilde{H}^{A,B}_\D(W_r) +O(1) \\
				&= \frac{1}{\alpha^*+\beta^*} \ P_\D\left((A\backslash \D)'\right),
			\end{split}
		\end{equation*}
		and so $\underset{\alpha^*\to\beta^*}{\text{Res}} (P_\D(A\backslash \D)) = P_{\D'\cup\{\beta^*\}}\left((A\backslash \D)'\right).$\\$\ \ $ \\
		
		
		\textit{\textbf{Case 3:}} We consider the case when $\alpha^*\in A\backslash \D, \beta^*\in \D.$ Then $A_{DL}(\D)$ is regular as $\alpha^*\to-\beta^*.$ Similar to the previous case, the only pole of $P_\D(A\backslash \D)$ comes when $W_r=\{\alpha^*\}$, i.e., when
		
		\begin{equation*}
			\widetilde{H}^{A,B}_\D = \begin{cases}
				\frac{1}{\alpha^*+\beta^*} +O(1), \ \text{if}, & W_r=\{\alpha^*\},\\
				O(1), & \text{otherwise}.
			\end{cases}
		\end{equation*}
		Therefore,
		
		\begin{equation*}
			\begin{split}
				P_\D(A\backslash \D) &= \frac{1}{\alpha^*+\beta^*} \ \widetilde{Q}(\D) \sum_{\substack{(A\backslash \D)'=\sum W_r \\ |W_r|\leqslant 2 }} \prod_r \widetilde{H}^{A,B}_\D(W_r) +O(1)\\
				&= \frac{1}{\alpha^*+\beta^*} \ P_\D\left((A\backslash \D)'\right),
			\end{split}
		\end{equation*}
		and $\underset{\alpha^*\to\beta^*}{\text{Res}}  (P_\D(A\backslash \D)) = P_{\D}\left((A\backslash \D)'\right).$\\$\ \ $ \\

		\textit{\textbf{Case 4:}} Finally, we consider  the case when $\alpha^*, \beta^*\in \D.$ In the previous section we have expansions for $A_\D^{A,B}(W_r),H_\D^{A,B}(W_r)$ and $A_{DL}(\D,\D,\D)$ in Lemmas \ref{expansion A_D^{A,B}}, \ref{expansion H_D^{A,B}} and \ref{expansion A_{DL}} respectively. It remains to find an expansion for $Q(\D)$ around $\alpha^*=\beta^*$.\\  
		First, we write
		
		\begin{equation*}\label{Q(D')}
			\begin{split}
				Q(\D) &= \left(-\log q\right)^{\#\D} |D|^{-\sum_{\delta\in \D}\delta} \prod_{\delta\in \D} \xx\left(\tfrac{1}{2}+\delta\right)   \sqrt{\frac{Z_{\z}(\D^{-},\D^{-}) Z_{\z}(\D,\D) Y_{\z}(\D^{-})  }{  Z_{\z}^\dag( \D^{-},\D)^2  Y_{\z}(\D)}} \\
				&= Q(\D') \ \left(\log q\right)^2 \ |D|^{-(\alpha^*+\beta^*)} \ q^{\alpha^*+\beta^*}\\
				&\ \ \ \times \prod_{\substack{\gamma\in  \D'}} \frac{\z(1-\alpha^*-\gamma) \z(1-\beta^*-\gamma) \z(1+\gamma+\alpha^*) \z(1+\gamma+\beta^*)}{\z(1-\alpha^*+\gamma) \z(1-\beta^*+\gamma) \z(1+\alpha^*-\gamma) \z(1+\beta^*-\gamma)}\\
				& \ \ \ \times \frac{\z(1-\alpha^*-\beta^*) \z(1+\alpha^*+\beta^*) \z(1-2\alpha^*) \z(1-2\beta^*)}{\z(1-\alpha^*+\beta^*)  \z(1-\beta^*+\alpha^*)}.
			\end{split}
		\end{equation*}
		Then, we have the expansion series for $Q(\D)$ around $\alpha^*=\beta^*$ given by
		
		\begin{equation}\label{Oterm}
			\begin{split}
				Q(\D)
				& = Q(\D') \left(\log q\right)^2\ \left(\frac{-1}{(\alpha^*+\beta^*)^2 \left(\log q\right)^2}+\frac{1}{12}+O\left((\alpha^*+\beta^*)^2\right)\right)\\
				&\ \ \ \times \Bigg[\Bigg(1 -\left(\alpha ^*+\beta ^*\right) \left(H_{\D'}^{A,B} (\{\beta\})+H_{\D'}^{A,B} (\{-\beta\})+\log |D| -\log q\right) \\
				&\ \ \ +O\left(\left(\alpha^*+\beta^*\right)^2\log|D|\right) \Bigg].
			\end{split}
		\end{equation}
		Combining these results we deduce from equation (\ref{Oterm}) that
		
		\begin{equation*}\label{case4}
			\begin{split}
				P_\D(A\backslash \D) 
				&=  -\frac{1}{(\alpha^*+\beta^*)^2} P_{\D'}\left(A\backslash \D\right)+\frac{1}{(\alpha^*+\beta^*) }  \Bigg(P_{\D'}\left(\left(A\backslash \D\right)\cup \{\beta^*\}\right)\\
				&\ \ \ + P_{\D'}\left(\left(A\backslash \D\right)\cup \{-\beta^*\}\right) -  \left(\log |D| -\log q\right) P_{\D'}\left(A\backslash \D\right)\Bigg) \\
				&\ \ \ +O\left((\alpha^*+\beta^*)^2\log|D|\right) +O\left(1\right).
			\end{split}
		\end{equation*}
		Thus, we have
		
		\begin{equation*}
			\begin{split}
				\underset{\alpha^*\to-\beta^*}{\mathrm{Res}}\left(P_\D\left(A \backslash \D\right)\right)
				& = P_{\D'}\left(\left(A\backslash \D\right)\cup \{\beta^*\}\right) + P_{\D'}\left(\left(A\backslash \D\right)\cup \{-\beta^*\}\right) \\
				&\ \ \ -  \left(\log |D| -\log q\right) P_{\D'}\left(A\backslash \D\right).
			\end{split}
		\end{equation*}
		Note that the term $\frac{1}{(\alpha^*+\beta^*)^2}$ will be canceled with the term $\frac{1}{(\alpha^*+\beta^*)^2}$ in the first case. Thus, combining the four cases we obtain the first part of the theorem. Namely,
		
		\begin{equation*}
			\begin{split}
				\underset{\alpha^*\to-\beta^*}{\mathrm{Res}}& \left(J_{DL}^*(A)\right)\\
				&= \sum_{\substack{\D\subseteq A\\\alpha^*,\beta^*\in \D}}    \Bigg(P_{\D'}\left(\left(A\backslash \D\right)\cup \{\beta^*\}\right) + P_{\D'}\left(\left(A\backslash \D\right)\cup \{-\beta^*\}\right)\\
				&\ \ \ -  \left(\log |D| -\log q\right) P_{\D'}\left(A\backslash \D\right)\Bigg) + \sum_{\substack{\D\subseteq A\\\alpha^*\in \D,\beta^*\in A\backslash \D}} P_{\D'\cup\{-\beta^*\}}\left(\left(A\backslash \D\right)'\right)\\
			\end{split}
		\end{equation*}
		\begin{equation*}
			\begin{split}
				\color{white}\underset{\alpha^*\to-\beta^*}{\mathrm{Res}}
				& \ \ \  + \sum_{\substack{\D\subseteq A\\\alpha^*\in A\backslash \D,\beta^*\in \D}} P_{\D}\left(\left(A\backslash \D\right)'\right)  \\
				&=  -\frac{\mathcal{X}'_D}{\mathcal{X}_D}\left(\tfrac{1}{2}+\beta^*\right) J_{DL}^*\left(A\right) + J_{DL}^*\left(A\cup\{\beta^*\}\right) +J_{DL}^*\left(A\cup\{-\beta^*\}\right),
			\end{split}
		\end{equation*}
		where
		
		\begin{equation*}
			\begin{split}
				\frac{\xx'_D}{\xx_D}(s) &= -\log|D|+ \log q.
			\end{split}
		\end{equation*} 
		Next, we consider the pole at zero. We split the proof into two cases.
		
		\begin{enumerate}
			\item $\alpha^*\in A\backslash \D.$
			\item $\alpha^*\in \D.$
		\end{enumerate}	
		
		\textit{\textbf{Case 1}} Suppose $\alpha^*\in A\backslash \D$. Then $Q(\D)$ is independent of $\alpha^*$, and $A_{DL}(\D)$ and $A_\D^{A,B}(W_r)$ are regular at zero. So the only pole of $P_\D(A\backslash \D)$ comes when $W_r=\{\alpha^*\},$ i.e., when
		
		\begin{equation*}
			\begin{split}
				H_\D^{A,B}\left(\{\alpha^*\}\right) &= \sum_{\gamma\in D} \left(\frac{\z'}{\z}(1+\alpha^*-\gamma)-\frac{\z'}{\z}(1+\alpha^*+\gamma)\right)-\frac{\z'}{\z}(1+2\alpha^*).
			\end{split}
		\end{equation*}
		Remember that the expansion for $\frac{\z'}{\z}(1+x) $ is $-\frac{1}{x}+O\left(1\right)$, so the residue as $\alpha^*\to 0$ of $-\frac{\z'}{\z}(1+x) $ is $1$. Therefore
		
		\begin{equation*}
			\widetilde{H}_\D(W_r)= \begin{cases}
				\frac{1}{2\alpha^*}+O(1)&, \ \text{if} \ W=\{\alpha^*\},\\
				O(1) &, \text{otherwise}.
			\end{cases}
		\end{equation*}
		So,
		
		\begin{equation*}
			\begin{split}
				P_\D(A\backslash \D) &= \widetilde{Q}(\D) \sum_{\substack{A\backslash \D=\sum_rW_r\\|W_r|\le 2}} \prod_r \widetilde{H}^{A,B}_\D(W_r) \\
				&= \frac{1}{2\alpha^*} \ \widetilde{Q}(\D) \sum_{\substack{(A\backslash \D)'=\sum_rW_r\\|W_r|\le 2}} \prod_r \widetilde{H}^{A,B}_\D(W_r) +O(1).
			\end{split}
		\end{equation*}
		Thus,
		
		\begin{equation*}
			\underset{\alpha^*\to 0}{\mathrm{Res}} \left(P_\D(A\backslash \D)\right) =\frac{1}{2} P_\D\left((A\backslash \D)'\right).
		\end{equation*}$ \ \ $\\
		
		\textit{\textbf{Case 2}} Suppose $\alpha^*\in \D$, then $H_\D^{A,B}(A\backslash \D), A_{DL}(\D)$ and $A^{A,B}_\D(W_r)$ are regular at $\alpha^*=0$. Also note that $H_\D^{A,B}(A\backslash \D)\vert_{\alpha^*=0} =H_{\D'}^{A,B}(A\backslash \D).$ Then, similar to equation (\ref{Q(D')}), we have
		
		\begin{equation*}
			\begin{split}
				Q(\D) &= Q(\D') \ \log q \ |D|^{-\alpha^*} \ \xx\left(\tfrac{1}{2}+\alpha^*\right)\ \z(1-2\alpha^*)\\
				&\ \ \ \times \prod_{\substack{\gamma\in  \D'}} \frac{\z(1-\alpha^*-\gamma)  \z(1+\gamma+\alpha^*) }{\z(1-\alpha^*+\gamma)  \z(1+\alpha^*-\gamma)}.
			\end{split}
		\end{equation*}
		One can see that the only pole comes from the term $\z(1-2\alpha^*)$, which has a residue of $-\frac{1}{2 \log q}$. Thus
		
		\begin{equation*}
			\begin{split}
				\underset{\alpha^*\to 0}{\mathrm{Res}} &\left(P_\D\left(A\backslash \D\right)\right) \\
				&= -\frac{1}{2 \log q} \widetilde{Q}(\D') \prod_{\substack{\gamma\in  \D'}} \frac{\z(1-\gamma)  \z(1+\gamma) }{\z(1+\gamma)  \z(1-\gamma)} \sum_{\substack{A\backslash \D=\sum_rW_r \\|W_r|\le 2}} \prod_r \widetilde{H}^{A,B}_\D(W_r)\\
				&= -\frac{1}{2 } P_{\D'}\left(A\backslash \D\right).
			\end{split}
		\end{equation*}
		Finally, combining the two results gives that
		
		\begin{equation*}
			\begin{split}
				\underset{\alpha^*\to 0}{\mathrm{Res}} \left(J^*_{DL}(\D)\right) &= \frac{1}{2} \sum_{\substack{\D\subseteq A\\\alpha^*\in A\backslash \D} }P_\D\left((A\backslash \D)'\right) -\frac{1}{2 } \sum_{\substack{\D\subseteq A\\\alpha^*\in  \D} } P_{\D'}\left(A\backslash \D\right)\\
				&= \frac{1}{2} \sum_{\substack{\D\subseteq A'} }P_\D\left((A\backslash \D)'\right) -\frac{1}{2 } \sum_{\substack{\D\subseteq A'} } P_{\D'}\left(A\backslash \D\right)\\
				&= 0,
			\end{split}
		\end{equation*}
		which completes the proof of the second part of the theorem \ref{J^*_{DL}.}.
	\end{proof}

	
	\section{$n$-level Density of Zeros of Families of Quadratic Dirichlet $L$-functions}\label{ndensity}
	
	In this section we derive the formula for the $n$--level density of the zeros of quadratic Dirichlet $L$-functions over function fields associated to $\x_D$, with $D$ being a square-free monic polynomial in $\mathbb{F}_{q}[T]$. 
	
	Let $\gamma_{i,D}$ denote the $i^{th}$ zero of $L(s,\x_D)$ on the half-line (note that, unlike in the number field case, we do not need to assume that all of the complex zeros are on the half-line since the Riemann hypothesis is true for this family of $L$-functions). Recall that $L(s,\x_D)$ is a function of $u=q^{-s}$, therefore is periodic with period $2\pi i/\log q,$ so we can confine our analysis of the zeros for the range $-\pi i/\log q \leqslant \mathfrak{I}(s) \leqslant \pi i/\log q$. 
	
	We define the $n$-correlation function, $S_n^{DL}(f)$ for the zeros for a family of $L$-functions,to be,
	
	\begin{equation*}
		S_n^{DL}(f) = \sum_ {D\in\h} \sum_{t_1,\cdots,t_n>0} f\left(\gamma_{t_1,D}, \cdots, \gamma_{t_n,D}\right),
	\end{equation*}
	where $f$ is an even $(2\pi/\log q)$-periodic test function and holomorphic.
	Note that the sum over the zeros in $S_n^{DL}(f)$ is not restricted to a sum over distinct indices at this stage.
	
	\begin{theorem}\label{S_n^{DL}}
		Let $C_{-}$ denote the path from $-c-\pi i/\log q$ up to  $-c+\pi i/\log q$ and let $C_{+}$ denote the path from $c-\pi i/\log q$ up to $c+\pi i/\log q$, where $3/4>c>1/2+1/\log|D|$. Let $f$ be an even $(2\pi/\log q)$--periodic test and holomorphic function of $n$ variables such that
		
		\begin{equation*}
			f\left( \theta_{j_1},\cdots,\theta_{j_n}\right) = f\left( \pm\theta_{j_1},\cdots,\pm\theta_{j_n}\right).
		\end{equation*}
		Then,
		
		\begin{equation*}
			\begin{split}
				&2^n\left(2\pi i\right)^n S_n^{DL}(f) \\
				&= \sum_{\substack{K\cup L\cup M \\=\{1,\cdots,n\}}} \int_{C^{K}_{+}} \int_{C^{L\cup M}_{-}} \left(-1\right)^{|M|} J_{DL}\left(z_K\cup -z_L,-z_M\right) f\left(iz_1,\cdots,iz_n\right) d_{z_1}\cdots d_{z_n}
			\end{split}
		\end{equation*}
		where $z_K=\{z_k:k\in K\},$ $-z_L=\{-z_l: l\in L\}$ and $-z_M=\{-z_m: m\in M\},$ $\int_{C^{K}_{+}} \int_{C^{L+M}_{-}}$ means we are integrating all the variables in $z_K$ along the $C_{+}$ path and all others along the $C_{-}$ path and define
		
		\begin{equation}\label{ourJ_{DL}(A,B)}
			J_{DL}(A,B) = \sum_{D\in\h} \prod_{\alpha\in A} \frac{L'}{L} \left(\frac{1}{2}+\alpha,\x_D\right)  \prod_{\beta\in B} \frac{\mathcal{X}'_D}{\mathcal{X}_D}\left(\tfrac{1}{2}-\beta,\x_D\right),
		\end{equation}
		where $\xx_D= |D|^{\frac{1}{2}-s}\xx(s)$ with $\xx(s) = q^{-\frac{1}{2}+s}.$ The equation above is the extended definition of $J_{DL}(A)$ that was previously defined in equation (\ref{dR_{A,B}}). The sets $K,L,M$ are finite sets of integers and the sets $A,B$ are finite sets of complex numbers.
	\end{theorem}
	
	\begin{proof}
		It follows from Cauchy's Residue Theorem that
		
		\begin{equation*}
			\begin{split}
				\left(2\pi i\right)^n & 2^n S_n^{DL}(f) \\
				& = \left(\int_{C_{+}} -\int_{C_{-}}\right)^n \sum_{D\in\h}\left(\frac{L'}{L}\left(\tfrac{1}{2}+\alpha_1,\x_D\right)\cdots \frac{L'}{L}\left(\tfrac{1}{2}+\alpha_n,\x_D\right) \right)\\
				& \ \ \  \times f\left(-i\alpha_1,\cdots,-i\alpha_n\right)d\alpha_1\cdots d\alpha_n.
			\end{split}
		\end{equation*}{\color{white}d\\}
		The term $2^n$ appears here because we are picking up poles from both the positive and negative zeros, and we are interested only on the terms where all the zeros are above the real axis. With this we have
		
		\begin{equation*}
			\begin{split}
				\left(2\pi i\right)^n & 2^n S_n^{DL}(f) \\
				& =\sum_{\substack{K\cup L= \{1,\cdots,n\}}} \left(-1\right)^{|L|} \int_{C_{+}^K} \int_{C_{-}^L} \sum_{D\in\h} \prod_{i\in K\cup L} \frac{L'}{L}\left(\tfrac{1}{2}+\alpha_i,\x_D\right)\\
				& \ \ \  \times f\left(-i\alpha_1,\cdots,-i\alpha_n\right)d\alpha_1\cdots d\alpha_n.
			\end{split}
		\end{equation*}{\color{white}d\\}
		Recalling the functional equation $\frac{L'}{L}\left(1-s,\x_D\right) = \frac{\xx'_D}{\xx_D}(s,\x_D)-\frac{L'}{L} (s,\x_D)$ and rewriting the terms being integrated over the $C_{-}$ contour we have
		
		\begin{equation}\label{161}
			\begin{split}
				\left(2\pi i\right)^n & 2^n S_n^{DL}(f)\\
				&= \sum_{\substack{K\cup L\cup M=\\\{1,\cdots,n\}}}\left(-1\right)^{|M|} \int_{C_{+}^K} \int_{C_{-}^{L\cup M}}\sum_{D\in\h} \prod_{k\in K} \frac{L'}{L}\left(\tfrac{1}{2}+z_k,\x_D\right)\\
				& \ \ \times \prod_{l\in L} \frac{L'}{L}\left(\tfrac{1}{2}-z_l,\x_D\right)  \prod_{m\in M} \frac{\mathcal{X}'_D}{\mathcal{X}_D}\left(\tfrac{1}{2}-z_m,\x_D\right)\\
				& \ \ \times f\left(-i z_1,\cdots, -iz_n\right) dz_1 \cdots dz_n.\\
				&= \sum_{\substack{K\cup L\cup M=\\\{1,\cdots,n\}}} \left(-1\right)^{|M|} \int_{C{+}^K} \int_{C_{-}^{L\cup M}}\sum_{D\in\h} J_{DL}\left(z_K\cup -z_L,-z_M\right)\\
				& \ \ \times f\left(-i z_1,\cdots, -iz_n\right) dz_1 \cdots dz_n.
			\end{split}
		\end{equation}
	\end{proof}
	
	{\color{white}d\\}
	
	Define,
	
	\begin{equation}\label{J*_DL}
		\begin{split}
			J^*_{DL}(A,B) &= \sum_{D\in\h} \prod_{\beta\in B} \frac{\xx'_D}{\xx_D}\left(\tfrac{1}{2}-\beta,\x_D\right) \sum_{\D\subseteq A} \left|D\right|^{- \sum_{\delta\in \D}\delta }   \\
			&\ \  \ \times \prod_{\delta\in \D} \xx\left(\tfrac{1}{2}+\delta\right) \sqrt{\frac{ Z_{\z}(\D^{-},\D^{-}) Z_{\z}(\D,\D) Y_{\z}(\D^{-})  }{  Z_{\z}^\dag( \D^{-},D)^2  Y_{\z}(\D)}} \\
			& \ \ \ \times \left( -\log q\right)^{\#\D} A_{DL}(\D,\D,\D)\\
			& \ \ \ \times \sum_{\substack{A\setminus \D = \sum_rW_r\\ |W_r|\leqslant 2}} \prod_{r=1}^R \widetilde{H}^{A,B}_\D(W_r),
		\end{split}
	\end{equation}
	where $A_{DL}$ and $\widetilde{H}^{A,B}_\D$ are defined in (\ref{A_{DL}1}) and (\ref{widetilde{H}^{A,B}_D}) respectively. Using Conjecture \ref{Our Ratios Conjecture2.} and Theorem \ref{J^*_{DL}.}, we can replace our $J_{DL}(A,B)$ in (\ref{ourJ_{DL}(A,B)}) with $J_{DL}^*(A,B)+o\left(|D|\right)$. Note that $J_{DL}^*(A)$ in (\ref{J^*_{DL}1}) is the equivalent to $J_{DL}^*(A,\emptyset)$ in our new notation. Thus with this we have
	
	\begin{equation}\label{integral}
		\begin{split}
			\left(2\pi i\right)^n & 2^n S_n^{DL}(f)\\
			&= \sum_{\substack{K\cup L\cup M=\\\{1,\cdots,n\}}} \left(-1\right)^{|M|} \int_{C{+}^K} \int_{C_{-}^{L\cup M}}\sum_{D\in\h} J^*_{DL}\left(z_K\cup -z_L,-z_M\right)\\
			& \ \ \times f\left(-i z_1,\cdots, -iz_n\right) dz_1 \cdots dz_n + o\left(|D|\right).
		\end{split}
	\end{equation}
	It is clear that the size of the error term is unaffected by integration with respect to $z_i$, which let us pull the error term outside the integral.
	
	The next step is to move all the contours integrals onto the imaginary axis. In order to do this, we recall from Section \ref{Residue Theorem for $J^*_{DL}(A)$} that for $\alpha^*,\beta^*\in A$ we have,
	
	\begin{equation*}
		\begin{split}
			\underset{\alpha^*\to -\beta^*}{\mathrm{Res}}&J^*_{DL}(A,B) \\
			&= J^*\left(A'\cup\{\beta^*\}\right) +J^*\left(A'\cup\{-\beta^*\}\right)+ J\left(A',B\cup \{-\beta^*\}\right),
		\end{split}
	\end{equation*}
	with $A'=A\backslash\{\alpha^*,\beta^*\}$. Now, for a given $n$ such that $0\leqslant R \leqslant n,$ we define
	
	\begin{equation}\label{n,R}
		\sum\nolimits^{n,R} = \sum_{\substack{j_1,\cdots,j_n = 1\\j_i\neq j_k \forall i,k>R}}^\infty  .
	\end{equation}
	Let $K,L,M$ be fixed sets such that $K\cup L\cup M =\{1,\cdots,n\}$, and $I_{f,K,L,M}^{n,R}$ be the integral in (\ref{integral}), excluding the error term, with $n-R$ of the integrals shifted onto the imaginary axis. All the integrals on the imaginary axis are principal value integrals, so we have
	
	\begin{equation*}
		\begin{split}
			&I_{f,K,L,M}^{n,R} \\
			&= \int_{-i\pi/\ln q}^{i\pi/\ln q} \cdots \int_{-i\pi/\ln q}^{i\pi/\ln q} \int_{C_{+}^{K\cap \{1,\cdots,R\}}} \int_{C_{-}^{(L\cup M)\cap \{1,\cdots,R\}}} J_{DL}^*\left(z_K\cup-z_L,-z_M\right)\\
			&\ \ \ \times f\left(iz_1,\cdots,iz_n\right)dz_1 \cdots dz_n.
		\end{split}
	\end{equation*}
	Hence, Theorem \ref{S_n^{DL}} can be expressed with the new notation, that is
	
	\begin{equation}\label{n=n}
		\begin{split}
			\left(2\pi i\right) 2^n \sum_{D\in\h} \sum\nolimits^{n,n} &f\left(\gamma_{j_1,D},\cdots, \gamma_{j_n,D}\right) \\
			& = \sum_{\substack{K\cup L\cup M=\\ \{1,\cdots,n\}}} \left(-1\right)^{|M|} I_{f,K,L,M}^{n,n} +o\left(|D|\right).
		\end{split}
	\end{equation}
	
	Now, we can state the following result.
	
	\begin{theorem}\label{prev}
		Assume Conjecture \ref{Our Ratios Conjecture2.}. Considering the notation defined above, we have
		
		\begin{equation}\label{234}
			\begin{split}
				\left(2\pi i\right)^n  2^n \sum_{D\in\h} \sum\nolimits^{n,R} &f\left(\gamma_{j_1,D}\cdots, \gamma_{j_n,D}\right)\\
				&= \sum_{\substack{K\cup L\cup M=\\ \{1,\cdots,n\}}} \left(-1\right)^{|M|} I_{f,K,L,M}^{n,R} +o\left(|D|\right).
			\end{split}
		\end{equation}
	\end{theorem}
	
	\begin{remark}
		The proof of Theorem \ref{prev} is similar to the number field case (see \cite[Theorem 6.5]{mason&Snaith}).
	\end{remark}
	
	We finish this section by stating our final result for the $n$-correlation of zeros of families of Dirichlet $L$-functions associated with the quadratic character $\x_D$.
	
	\begin{theorem}[$n$-level Density of Zeros of families of Quadratic Dirichlet $L$-functions] \label{n-density theorem}
		Assume Conjecture \ref{Our Ratios Conjecture2.}. Then
		
		\begin{equation*}
			\begin{split}
				\sum_{D\in\h} &\sum\nolimits^{n,0} f\left(\gamma_{j_1,D},\cdots,\gamma_{j_n,D}\right) \\
				& =\frac{1}{\left(2\pi\right)^n} \sum_{\substack{K\cup L\cup M \\=\{1,\cdots,n\}}} \left(-1\right)^{|M|} \left(\int_0^{\pi /\ln q}\right)^n J_{DL}^*\left(-iz_K\cup iz_L,iz_M\right) \\
				& \ \ \ \times f\left(z_1,\cdots,z_n\right) dz_1\cdots dz_n + o\left(|D|\right).
			\end{split}
		\end{equation*}
	\end{theorem}
	
	\begin{proof}
		Directly from Theorem \ref{prev}, for $R=0,$ we have that
		
		\begin{equation*}
			\begin{split}
				(2\pi i)^n  2^n \sum_{D\in\h} &\sum\nolimits^{n,0} f\left(\gamma_{j_1,D},\cdots,\gamma_{j_n,D}\right)\\
				& \approx \sum_{\substack{K\cup L\cup M \\=\{1,\cdots,n\}}} \left(-1\right)^{|M|} I_{f,K,L,M}^{n,0} \\
				&= \sum_{\substack{K\cup L\cup M \\=\{1,\cdots,n\}}} \left(-1\right)^{|M|}  \left(\int_{-i\pi/\ln q}^{i\pi/\ln q}\right)^n J^*_{DL} \left(z_K\cup-z_L,-z_M\right)\\
				& \ \ \ \times f\left(iz_1,\cdots,iz_n\right)dz_1\cdots dz_n,
			\end{split}
		\end{equation*}
		where the integral above is a principal value integral. Changing the variable to $\gamma_{j,D}=-iz_j$ for all $j$, we have
		
		\begin{equation*}
			\begin{split}
				&\left(2\pi i\right)^n  2^n \sum_{D\in\h} \sum\nolimits^{n,0} f\left(\gamma_{j_1,D},\cdots,\gamma_{j_n,D}\right)\\
				& \approx  \sum_{\substack{K\cup L\cup M \\=\{1,\cdots,n\}}} \left(-1\right)^{|M|}  \left(\int_{-\pi/\ln q}^{\pi/\ln q}\right)^n J^*_{DL} \left(-iz_K\cup iz_L,iz_M\right)f\left(z_1,\cdots,z_n\right)dz_1\cdots dz_n\\
				&=   \left(\int_{-\pi/\ln q}^{\pi/\ln q}\right)^n \sum_{\substack{K\cup L\cup M \\=\{1,\cdots,n\}}} \left(-1\right)^{|M|} J^*_{DL} \left(-iz_K\cup iz_L,iz_M\right) f\left(z_1,\cdots,z_n\right)dz_1\cdots dz_n.
			\end{split}
		\end{equation*}
		Since we have $\frac{\xx'_D}{\mathcal{X}_D}\left(\tfrac{1}{2}+\alpha\right) = \frac{\xx'_D}{\xx_D}\left(\tfrac{1}{2}-\alpha\right)$, independent of $\alpha$, and $f\left( \theta_{j_1},\cdots,\theta_{j_n}\right) = f\left( \pm\theta_{j_1},\cdots,\pm\theta_{j_n}\right),$ is an even function, we can write the sum above as
		
		\begin{equation*}
			\begin{split}
				\sum_{\substack{K\cup L\cup M=\\\{1,\cdots,n\}}} \left(-1\right)^{|M|} J_{DL}^*&\left(-iz_K\cup iz_L,iz_M\right)\\
				& = \left(-1\right)^{|M|} J_{DL}^*\left(iz_K\cup -iz_L,-iz_M\right).
			\end{split}
		\end{equation*}
		Thus, each integral from $-\pi/\log q$ to $\pi/\log q$ is a double of the integral from $0$ to $\pi/\ln q$. Hence,
		
		\begin{equation*}
			\begin{split}
				\left(2\pi\right)^n& 2^n \sum_{D\in\h} \sum\nolimits^{n,0} f\left(\gamma_{j_1,D},\cdots, \gamma_{j_n,D} \right)\\
				&= 2^n \left(\int_0^{\pi/\ln q} \right)^n \sum_{\substack{K\cup L\cup M \\=\{1,\cdots,n\}}} \left(-1\right)^{|M|} J_{DL}^*\left(-iz_K\cup iz_L,iz_M\right) \\
				& \ \ \ \times f\left(z_1,\cdots,z_n\right) dz_1 \cdots dz_n +o(|D|).
			\end{split}
		\end{equation*}
	\end{proof}
	
	The only thing that remains to show is that there are no poles on the path of integration. Since $f$ is holomorphic, we need only to check that the sum
	
	\begin{equation*}
		\sum_{\substack{K\cup L\cup M=\\\{1,\cdots,n\}}} \left(-1\right)^{|M|} J_{DL}^*\left(-iz_K\cup iz_L,iz_M\right)
	\end{equation*}
	has no poles at $z_1=z_2.$ We do not need to check for the pole at $z_1=-z_2$ as $z_i\geqslant 0, $ for $0\leqslant i \leqslant n$ on the path of integration. We have that $J_{DL}^*\left(-iz_K\cup iz_L,iz_M\right)$ has a pole at $z_1=z_2$ if $1\in K,2\in L$ or $2\in K, 1\in L.$ Therefore,
	
	\begin{equation*}
		\begin{split}
			\underset{z_1=z_2}{\mathrm{Res}} &\left(\sum_{\substack{K\cup L\cup M \\=\{1,\cdots,n\}}} \left(-1\right)^{|M|} J_{DL}^*\left(-iz_K\cup iz_L,iz_M\right)\right)\\
			&= \sum_{\substack{K\cup L\cup M \\=\{1,\cdots,n\}}} \left(-1\right)^{|M|} \underset{z_1=z_2}{\mathrm{Res}}\Big(J_{DL}^*\left(-iz_{K\cup\{z_1\}}\cup iz_{L\cup\{z_2\}},iz_M\right) \\
			& \ \ \ + J_{DL}^*\left(-iz_{K\cup\{z_2\}}\cup iz_{L\cup\{z_1\}},iz_M\right)\Big)\\
			&= 0
		\end{split}
	\end{equation*}
	since $ \mathrm{Res}_{x=y}f(x,y)=-\mathrm{Res}_{x=y}f(-x,-y)$. The same argument holds for all poles at $z_i=\pm z_j$ for all $0\leqslant i,j\leqslant n.$ Thus, if $\sum_{K\cup L\cup M=\{1,\cdots,n\}} \left(-1\right)^{|M|} $ $J_{DL}^*\left(-iz_K\cup iz_L\right)$ has a singular set, it has dimension of $n-1$ or smaller. But from the theory of several complex variables, this implies the singular set is trivial \cite{krantz}. Therefore, there cannot be any poles on the path of integration.

	
	\section{An Example of Calculation}\label{1-level}

	In this section, we calculate the one-level density of zeroes of the family of $L$-functions associated with $\x_D$ to exemplify the n-level calculation carried out in Section \ref{ndensity}. We compare our calculation in this section with those of Andrade and Keating appearing in Theorem \ref{andrade&keating n=1}.

	\subsection{Evaluate $S_1^{DL}$}\label{Evaluate $S_1^{DL}$}
	
	Let $C_{-}$ denote the path from $-c-\pi i/\log q$ up to $-c+\pi i/\log q$ and  $C_{+}$ denote the path from $c-\pi i/\log q$ up to $c+\pi i/\log q$, where $3/4>c>1/2+1/\log|D|$. Let $f$ be an even $(2\pi/\log q)$--periodic and holomorphic test function of one variable such that
	
	\begin{equation*}
		f(\theta) =f(-\theta),
	\end{equation*}
	then,
	
	\begin{equation}\label{S_1^{DL}}
		S_1^{DL} = \sum_{D\in\h} \sum_{j>0} f\left(\gamma_{j,D}\right),
	\end{equation}
	where $\gamma_{j,D}$ is the $j^{\text{th}}$ zero of the $L$-function $L(s,\x_D)$. Hence, follows from Theorem \ref{n-density theorem} that
	
	\begin{equation*}
		\begin{split}
			2 &(2\pi i)  S_1^{DL}(f)\\
			&= \left(\int_{C_{+}} - \int_{C_{-}}\right) \sum_{D\in\h} \frac{L'}{L}\left(\tfrac{1}{2}+\alpha,\x_D\right) f(-i\alpha) d\alpha\\
			&= \sum_{\substack{K\cup L\cup M\\=\{1\}}} \int_{C_{+}^K} \int_{C_{-}^{L\cup M}} \left(-1\right)^{|M|} J_{DL}\left(z_K\cup -z_L,- z_M\right) f(-iz_1) dz_1\\
			&= \sum_{\substack{K\cup L\cup M\\=\{1\}}} \int_{C_{+}^K} \int_{C_{-}^{L\cup M}} \left(-1\right)^{|M|} J_{DL}^*\left(z_K\cup -z_L,-z_M\right) f(-iz_1) dz_1 +o\left(|D|\right).
		\end{split}
	\end{equation*}
	We move the integration onto the imaginary axis as there are no poles when we only have a single variable. Thus,
	
	\begin{equation*}
		\begin{split}
			2&(2\pi) S_1^{DL}\\
			&= \int_{-\pi/\log q}^{\pi/\log q} f(z) \Big[ \Big(J_{DL}^*(iz,\emptyset)+ J_{DL}^*(-iz,\emptyset)\Big) - J_{DL}^*(\emptyset,-iz)\Big]dz + o(|D|),
		\end{split}
	\end{equation*}
	where $J_{DL}^*(A,B)$ is defined as
	
	\begin{equation*}
		\begin{split}
			J^*_{DL}(A,B) &= \sum_{D\in\h} \prod_{\beta\in B} \frac{\xx'_D}{\xx_D}\left(\tfrac{1}{2}-\beta\right) \sum_{\D\subseteq A} \left|D\right|^{- \sum_{\delta\in \D}\delta }   \\
			&\  \ \ \times \prod_{\delta\in \D} \xx\left(\tfrac{1}{2}+\delta\right) \sqrt{\frac{ Z_{\z}(\D^{-},\D^{-}) Z_{\z}(\D,\D) Y_{\z}(\D^{-})  }{  Z_{\z}^\dag( \D^{-},\D)^2  Y_{\z}(\D)}} \\
			& \ \ \ \times\left( -\log q\right)^{\#\D} A_{DL}(\D,\D,\D) \sum_{A\setminus \D = W_1+\cdots+W_R} \prod_{r=1}^R \widetilde{H}^{A,B}_\D(W_r),
		\end{split}
	\end{equation*}
	where
	
	\begin{equation*}
		\begin{split}
			\widetilde{H}^{A,B}_\D (W_r)
			&= H^{A,B}_\D(W_r)+A^{A,B}_\D(W_r),
		\end{split}
	\end{equation*}
	with
	
	\begin{equation*}
		\begin{split}
			H^{A,B}_\D(W_r) &= \begin{cases}
				\underset{\beta\in \D}{\sum} \left(\frac{\z'}{\z}(1+\gamma-\beta) - \frac{\z'}{\z}(1+\gamma+\beta)\right)+\frac{\z'}{\z}(1+2\gamma) &, W_r=\{\gamma\}\\
				\left(\frac{\z'}{\z}\right)' (\gamma_1+\gamma_2) &, W_r=\{\gamma_1,\gamma_2\}\\
				0 &, |W_r|\ge 3,
			\end{cases}
		\end{split}
	\end{equation*}
	\begin{equation*}
		A^{A,B}_\D(W_r) = \prod_{\alpha\in W_r} \frac{\partial}{\partial \alpha} \log A_{DL} (A,B,\D) \Bigg\vert_{B=A},
	\end{equation*}
	and $A_{DL}(A,B,\D),Z_{\z}(A,B,\D),Y_{\z}(A,B,\D)$ are defined in equations (\ref{Zz1})--(\ref{Y_{DL}1}).
	
	
	\subsection{Checking $S^{DL}_1$ against Andrade and Keating result}
	
	Recall, Andrade and Keating's one-level density in Theorem \ref{andrade&keating n=1} is given by

	\begin{equation*}
		\begin{split}
			S_1(f)&= \sum_{D\in\h} \sum_{\gamma_D} f(\gamma_D) =  \frac{1}{2\pi} \int_{-\pi/\ln q}^{\pi/\ln q} f(z) \sum_{D\in \h} \Bigg[ \log|D|+\frac{\xx'}{\xx}(\frac{1}{2}-iz) \\
			& \ \ \ + 2  \Bigg(  \frac{\z'}{\z}(1+2iz)+ A_D'(iz;iz) - \left(\ln q\right) |D|^{iz} \xx\left(\tfrac{1}{2}+iz\right) \z(1-2iz)  \\
			& \ \ \ \times A_D(-iz,iz)\Bigg) \Bigg]dz + o\left(|D|\right),
		\end{split}
	\end{equation*}
	where $\gamma_D$ is the ordinate of a generic zero of $L(s,\x_D)$ and $f$ is an even and periodic suitable test function, and $A_D'(r,r)$ and $A_D(-r,r)$ are defined in equation (\ref{AD}). 
	
	One should note that $S_1(f)$ is different from our definition of $S_1^{DL}(f)$ given in equation (\ref{S_1^{DL}}), where $S_1(f)$ is a sum over all zeros of the given $L$-function while $S_1^{DL}(f)$ is a sum only over the positive zeros of the given $L$-function. We expect $S_1(f)=2S_1^{DL}(f).$ 
	
	Its easy to see that the result matches perfectly with $S_1^{DL}(f)$. The term containing $\log|D|-\frac{\xx'}{\xx}\left(\frac{1}{2}-iz\right)$ corresponds exactly to $J_{DL}^*(\emptyset,iz)$.
	
	If we consider the parts of $S_1^{DL}(f)$ corresponding to $J_{DL}^*\left(iz,\emptyset\right) +J_{DL}^*\left(-iz,\emptyset\right)$, we have that
	
	\begin{equation*}
		\begin{split}
			\int_{-\pi/\log q}^{\pi/\log q} f(z) & J_{DL}^*(iz,\emptyset) dz +\int_{-\pi/\log q}^{\pi/\log q} f(z)  J_{DL}^*(-iz,\emptyset) dz\\
			&=  2  \int_{-\pi/\log q}^{\pi/\log q} f(z)  \sum_{D\in\h} J_{DL}^*(iz,\emptyset)  dz,
		\end{split}
	\end{equation*}
	where we make a change of variable, letting $z\to-z$ in the second integral and use the fact that $f(z)$ is an even function. 
	
	Now, for the parts of $S_1^{DL}(f)$ corresponding to $J_{DL}^*\left(iz,\emptyset\right) +J_{DL}^*\left(-iz,\emptyset\right)$ when $\D=\emptyset,$ we have that
	
	\begin{equation*}
		\begin{split}
			\int_{-\pi/\log q}^{\pi/\log q} f(z)  & \left(J_{DL}^*(iz,\emptyset)+ J_{DL}^*(-iz,\emptyset)\right) dz\\
			& =  2 \int_{-\pi/\log q}^{\pi/\log q} f(z)  \sum_{D\in\h} \left(\frac{\z'}{\z}(1+2iz)+A'_D(iz,iz)\right) dz.
		\end{split}
	\end{equation*}
	
	On the other hand, for the parts of $S_1^{DL}(f)$ corresponding to $J_{DL}^*\left(iz,\emptyset\right) +J_{DL}^*\left(-iz,\emptyset\right)$ when $\D\neq\emptyset,$ we have
	
	\begin{equation*}
		\begin{split}
			&\int_{-\pi/\log q}^{\pi/\log q} f(z)  \left(J_{DL}^*(iz,\emptyset)+ J_{DL}^*(-iz,\emptyset)\right) dz \\
			&\ \ \ =  -2\log q \int_{-\pi/\log q}^{\pi/\log q} f(z)  \sum_{D\in\h} \Big(|D|^{-i z} \xx\left(\tfrac{1}{2}+iz\right) \z(1-2iz)A_D(-iz,iz) \Big) dz.\\
		\end{split}
	\end{equation*}
	
	Combining all the above results, we have that
	
	\begin{equation*}
		\begin{split}
			&\int_{-\pi/\log q}^{\pi/\log q}  f(z)   \left(J_{DL}^*(iz,\emptyset)+ J_{DL}^*(-iz,\emptyset)- J^*_{DL}(\emptyset,-iz\right) dz \\
			&\ \ \ = \int_{-\pi/\log q}^{\pi/\log q} f(z)  \sum_{D\in\h} \Bigg[ \log|D|+\frac{\xx'}{\xx}(\frac{1}{2}-iz) + 2  \Bigg(  \frac{\z'}{\z}(1+2iz)+ A_D'(iz;iz) \\
			& \ \ \ \ \ \  - \left(\ln q\right) |D|^{iz} \xx\left(\tfrac{1}{2}+iz\right) \z(1-2iz) A_D(-iz,iz)\Bigg)   \Bigg]dz.
		\end{split}
	\end{equation*}
	
	Hence, we can see that our result in Theorem \ref{n-density theorem} when $n=1$ agrees with the result of Theorem \ref{andrade&keating n=1}.

	\vspace{1.0cm}
	
	\noindent \textbf{Acknowledgment:} The first author is grateful to the Leverhulme Trust (Grant No. RPG-2017-320) for the support through the research project grant “Moments of $L$-functions in Function Fields and Random Matrix Theory”. The research of the second author is supported by the Public Authority for Applied Education and Training of Kuwait (PAAET).
\\

\noindent For the purpose of open access, the authors have applied a Creative Commons Attribution (CC BY) licence to any Author Accepted Manuscript version arising from this submission.

	
\end{document}